\documentclass[12pt, a4paper]{article}
\usepackage[utf8]{inputenc}

\usepackage{graphicx}
\usepackage{lscape}
\usepackage{amsfonts}
\usepackage{amssymb,amsthm,amsmath,color,mathrsfs}
\usepackage[T1]{fontenc}
\usepackage{lmodern}
\usepackage{hyperref}
\usepackage{tikz-cd}
\usepackage{mathtools}

\usepackage{amsmath,amssymb}

\newtheorem{Th}{Theorem}[section]
\newtheorem{Cor}[Th]{Corollary}
\newtheorem{Lem}[Th]{Lemma}
\newtheorem{Prop}[Th]{Proposition}
\newtheorem{Res}[Th]{Result}

\theoremstyle{definition}
\newtheorem{Def}[Th]{Definition}
\newtheorem{Not}[Th]{Note}
\newtheorem{Rem}[Th]{Remark}
\newtheorem{Ex}[Th]{Example}

\theoremstyle{plain}

\newcommand{\R}{\mathbb R}
\newcommand{\N}{\mathbb N}

\newcommand{\K}{\mathbb K}

\newcommand{\X}{\mathcal X}

\newcommand{\com}[2]{\mathscr C_{#2}(#1)}

\newcommand{\norm}{\mathcal{N}(\X)}
\newcommand{\notimplies}{%
	\mathrel{{\ooalign{\hidewidth$\not\phantom{=}$\hidewidth\cr$\implies$}}}}

\title{Comparability of Metrics and Norms in terms of Basis of Exponential Vector Space}

\author{Dhruba Prakash Biswas\footnote{Department of Pure Mathematics, University of Calcutta, 35, Ballygunge Circular Road, Kolkata-700019, India, e-mail : dhrubaprakash28@gmail.com}, Priti Sharma\footnote{Bangabasi College, 19, Rajkumar Chakraborty Sarani, Kolkata - 700009, India, e-mail : mspriti23@gmail.com}, Sandip Jana\footnote{Department of Pure Mathematics, University of Calcutta, 35, Ballygunge Circular Road, Kolkata-700019, India, e-mail : sjpm@caluniv.ac.in} \& Jens Schwaiger\footnote{Institut für Mathematik, Karl-Franzens-Universität Graz, Austria, e-mail : jens.schwaiger@uni-graz.at} } 

\date{}

\begin{document}
	
\maketitle
	
\begin{abstract}
	
In this paper, we shall compare two metrics in terms of `\textit{orderly dependence}’, a notion developed in exponential vector space in the article `\textit{Basis and Dimension of Exponential Vector Space}' by  Jayeeta Saha and Sandip Jana in \textit{Transactions of A. Razmadze Mathematical Institute} Vol. 175 (2021), issue 1, 101-115. Exponential vector space (in short `evs') is a partially ordered space associated with a commutative semigroup structure and a compatible scalar multiplication. In the present paper we shall show that the collection $ \mathcal{D(\mathbf X)} $ of all metrics on a non-empty set $ \mathbf X $, together with the constant function zero `$ O $', forms a topological exponential
vector space. We shall discuss the orderly dependence of two metrics through our findings of a basis of $ \mathcal{D(\mathbf X)}\smallsetminus\{O\} $ in different scenario. We shall  characterise `\textit{orderly independence}' of two elements of a topological evs
in terms of the `\textit{comparing function}', another mechanism developed in topological exponential vector space, which can measure the degree of comparability of two elements of an evs. Finally, we shall discuss existence of orderly independent norms on a linear space. For an infinite dimensional
linear space we shall construct a large number of orderly independent norms depending on the dimension of the linear space. Orderly independent norms are precisely those which are `\textit{totally non-equivalent}', in the sense that they produce incomparable topologies.
	\end{abstract}
	
	
	
\section{Introduction}

Two metrics on a set are said to be \textit{equivalent} if they produce the same topology. In the present paper we shall compare two metrics in terms of basis of \textit{exponential vector space}. The study of exponential vector space was initiated by S. Ganguly, S. Mitra and S. Jana in the paper \cite{C(X)} with the name `\textit{quasi vector space}'. Later the term `\textit{exponential vector space}' was coined by Priti Sharma and Sandip Jana in \cite{evs}. Let us first present the definition of exponential vector space.

		\begin{Def}\cite{evs}
		Let $(X,\leq)$ be a partially ordered set, `$+$' be a binary operation on 
		$X$ [called \emph{addition}] and `$\cdot$'$:K\times X\longrightarrow X$ be another composition [called \emph{scalar multiplication}, $K$ being a field]. If the operations and the partial order satisfy the following axioms then $(X,+,\cdot,\leq)$ is called an \emph{exponential vector space} (in short \emph{evs}) over $K$ .
		
		$A_1:$ $(X,+)\text{ is a commutative semigroup with identity } \theta$

		$A_2:$ $x\leq y\, (x,y\in X)\Rightarrow x+z\leq y+z  \text{ and } \alpha\cdot 
		x\leq
		\alpha\cdot y, \forall\, z\in X,  \forall\, \alpha\in K$

		$A_3:$ $\forall\,x,y\in X,\  \forall\,\alpha,\beta\in K$ 
		
		\hspace*{0.80cm}(i) $\alpha\cdot(x+y)=\alpha\cdot x+\alpha\cdot y$

		\hspace*{0.80cm}(ii)  $\alpha\cdot(\beta\cdot x)=(\alpha\beta)\cdot x$

		\hspace*{0.80cm}(iii) $(\alpha+\beta)\cdot x \leq \alpha\cdot x+\beta \cdot x$

		\hspace*{0.80cm}(iv) $1\cdot x=x,\text{ where `1' is the multiplicative identity in }K$

		$A_4:$ $\alpha\cdot x=\theta \text{ iff }\alpha=0\text{ or }x=\theta$

		$A_5:$ $x+(-1)\cdot x=\theta\text{ iff } x\in X_0:=\big\{z\in X:\ y\not\leq z, \forall\,y \in X\smallsetminus\{z\}\big \}$

		$A_6:$ For each $x \in X, \exists\,p\in X_0\text{ such that }p\leq x$. \label{d:evs}
	\end{Def}
	
	In the above definition, $X_0$ is precisely the set of all minimal elements of the evs $X$ with respect to the partial order on $X$ and it forms the maximum vector space (within $X$) over the same field as that of $X$ (\cite{C(X)}). This vector space $X_0$ is called the `\emph{primitive space}' or `\emph{zero space}' of $X$ and the elements of $X_0$ are called the `\emph{primitive elements}' [\cite{evs}]. 
	
	Thus every evs contains a vector space and conversely, given any vector space $V$, an evs $X$ can be constructed such that $V$ is isomorphic to $X_0$ \cite{evs}. In this sense,``exponential vector space'' can be considered as an algebraic ordered extension of vector space. The axiom $ A_3 $(iii) expresses very rapid growth of the non-primitive elements, since $x\leq \frac {1}{2} x +\frac{1}{2}x$, $\forall x \not \in \thinspace X_0$; whereas axiom $ A_6 $ demonstrates `\textit{positivity}' of all elements with respect to primitive elements. This justifies the nomenclature `exponential vector space'.
	
\noindent\textbf{Notation:} For any set $ A\subseteq X,\ X $ being an evs, we shall use the following notations:\\ \centerline{$\uparrow A:=\{x\in X: x\geq a\text{ for some }a\in A\}$ and $\downarrow A:=\{x\in X: x\leq a\text{ for some }a\in A\}$.}

The following concepts will be useful in the sequel.

	\begin{Def}{\cite{Nach}}
		Let `$\leq$' be a preorder in a topological space $Z$; the preorder is said to be \textit{closed} if its graph $G_{\leq}(Z):=\{(x,y)\in Z\times Z: x\leq y\}$ is closed in $Z\times Z$ (endowed with the product topology).
	\end{Def}
	
	\begin{Th}{\em{\cite{Nach}}}
		A partial order `$\leq$' in a topological space $Z$ will be a closed order iff for any $x,y\in Z$ with $x\not\leq y$, $\exists$ open neighbourhoods $U,V$ of $x,y$ respectively in $Z$ such that $(\uparrow U)\cap(\downarrow V)=\emptyset$.\label{t:partcl}
	\end{Th}
	
	\begin{Def}\cite{evs}
		An exponential vector space $X$ over the field $\K$ of real or complex numbers is said to be a \textit{topological exponential vector space} if there exists a topology on $X$ with respect to which the addition, scalar multiplication are continuous and the partial order `$\leq$' is closed (Here $\K$ is equipped with the usual topology).
	\end{Def}
	\begin{Rem}
		If $X$ is a topological exponential vector space over the field $\K$ then its primitive space $X_0$ becomes a topological vector space over $\K$, since restriction of a continuous function is continuous. Moreover, the closedness of the partial order `$\leq$' in a topological exponential vector space $X$ readily implies (in view of Theorem \ref{t:partcl}) that $X$ is Hausdorff and hence $X_0$ becomes a Hausdorff topological vector space.\label{rm:topvec}
	\end{Rem}

	\begin{Ex} \cite{mor}
		Let $X:=$ [$0,\infty$)$\times V$, where $V$ is a vector space over the field $\mathbb K$ of real or complex numbers. Define operations and partial order on $X$ as follows :
		for $(r,a),(s,b) \in X $ and $\alpha \in \mathbb K $,\\
		(i) $(r,a)+(s,b) := (r+s,a+b)$\\
		(ii) $\alpha (r,a) := (\lvert\alpha\rvert
		r,\alpha a)$\\
		(iii) $(r,a)\leq (s,b)$ iff $r\leq s$ and $a = b$\\
		Then [$0,\infty$)$\times V$ becomes an exponential vector space with the primitive space \{$0$\}$\times V$ which is clearly isomorphic to $V$.
		
		In this example, if we consider $V$ as a Hausdorff topological vector space, then [$0,\infty$)$\times V$ becomes a topological exponential vector space with respect to the product topology, where [$0,\infty$) is equipped with the subspace topology inherited from the real line $\mathbb R$.
		
		Instead of $V$, if we take the trivial vector space $ \{\theta\} $ in the above example, then the resulting topological evs is [$0,\infty$)$\times\{ \theta\}$ which can be clearly identified with the half ray [$0,\infty$) of the real line. Thus, [$0,\infty$) forms a topological evs over the field $\mathbb K$.\qed\label{e:vect}
	\end{Ex}
	
	\begin{Ex}{\cite{C(X)}}
		Let $\mathscr{C}(\X)$ denote the topological hyperspace consisting of all non-empty compact subsets of a Hausdorff topological vector space $\X$ over the field $\K$ of real or complex numbers. Define addition, scalar multiplication and partial order on $\mathscr{C}(\X)$ as follows:
		
		(i) For $A,B\in\mathscr{C}(\X)$, $A+B:=$ $\big\{a+b:$ $a\in A$, $b\in B\big\}$
		
		(ii) For $A\in\mathscr{C}(\X)$ and $\alpha\in\K$, $\alpha A:=$ $\big\{\alpha a:$ $a\in A\big\}$
		
		(iii) For $A,B\in\mathscr{C}(\X)$, $A\leq B\iff$ $A\subseteq B$
		
		Then $\mathscr{C}(\X)$ becomes an evs over the field $\K$. The primitive space is given by  $[\mathscr{C}(\X)]_{0}$ $=\big\{\{x\}:$ $x\in \X\big\}$.	Moreover, $\mathscr{C}(\X)$ forms a topological evs with respect to the Vietoris topology \cite{Mi}. An arbitrary basic open set in this topology is of the form $V_0^{+}\cap V_{1}^{-}\cap V_{2}^{-}\cap...\cap V_{m}^{-}$, where  $V_0,V_1,...,V_m$ are open in $\X$ with $V_i\subseteq V_0$ for all $i=1,2,...,m$. Here $ V^+:=\big\{A\in\mathscr{C}(\X):A\subseteq V\big\} $, $ V^-:=\big\{A\in\mathscr{C}(\X):A\cap V\neq\emptyset\big\} $, for any $ V\subseteq\X $.\qed\label{ex:vietor}
	\end{Ex}

	The algebraic structure of $\com{\X}{}$ was the primary motivation for the study of exponential vector space. Later a large number of examples were found in various discipline of Mathematics where the hyperspace structure is well preserved (see \cite{bal}, \cite{basis}, \cite{JTh}, \cite{norm}, \cite{spri}, \cite{mor}, \cite{qvs}, \cite{evs}, \cite{PTh}).  
	
	In the present paper we shall show first that the collection $\mathcal{D}(\mathbf{X})$ of all metrics on a non-empty set $\mathbf{X}$, together with the constant function zero `$O$', forms a topological exponential vector space with respect to suitably defined operations, partial order and topology of pointwise convergence. We shall show that $\mathcal{D}(\R^2)$ is an incomplete uniform space. 
	
	In section 4, we shall compare two metrics in terms of ‘\textit{orderly dependence}’, a notion developed in exponential vector space in the article  `\textit{Basis and Dimension of Exponential Vector Space}' \cite{basis}. We shall discuss the existence of basis of $ \mathcal{D(\mathbf X)} $ in different scenario. 
	
	Section 5 deals with characterisation of `\textit{orderly independence}' of two elements of a topological evs in terms of the `\textit{comparing function}', another mechanism developed in topological exponential vector space in the paper \cite{bal}, which can measure the degree of comparability of two elements of a topological evs.
	
	In the last section, we shall use the techniques developed in the previous sections to	investigate the comparability of two norms on a linear space. We shall show that the existence of \textit{orderly independent} norms depends on the dimension of the linear space. For an infinite dimensional linear space we shall construct a large number of orderly independent norms depending on the dimension of the linear space. Orderly independent norms are precisely those which are `\textit{totally non-equivalent}', in the sense that they produce incomparable topologies.

\section{Prerequisites}
	
\begin{Def} \cite{spri} A subset $Y$ of an exponential vector space $X$ is said to be a \emph{sub exponential vector space} 
	(\emph{subevs} in short) if $Y$ itself is an exponential vector space with respect to the compositions of 
	$X$ being restricted to $Y$.
	\label{d:subevs}\end{Def} 
	
		\begin{Def}\cite{mor}
		A map $\phi:X\to Y$ ($X,Y$ being two exponential vector spaces over the field $K$) is called an \textit{order-morphism} if
		
		(i) $f(x+y)=f(x)+f(y)$, $\forall\, x,y\in X$
		
		(ii) $f(\alpha x)=\alpha f(x)$, $\forall\alpha\in K$, $\forall\, x\in X$
		
		(iii) $x\leq y\ (x,y\in X)$ $\implies f(x)\leq f(y)$
		
		(iv) $p\leq q\ \big(p,q\in f(X)\big)$ $\implies f^{-1}(p)\subseteq\downarrow f^{-1}(q)$ and $f^{-1}(q)\subseteq\uparrow f^{-1}(p)$.
		
		Further, if $f$ is bijective (injective, surjective) order-morphism, then it is called \textit{order-isomorphism} (\textit{order-monomorphism}, \textit{order-epimorphism} respectively).
		
		If $X$ and $Y$ both are topological evs over the field $\K$, then the order-isomorphism $f:X\to Y$ is called \textit{topological order-isomorphism} if $f$ is a homeomorphism.
		\label{d:mor}\end{Def}
	
	\begin{Def} \cite{evs} A property of an evs is called an \emph{evs property} if it remains invariant under order-isomorphism.
	\end{Def}
	
	\begin{Prop}{\em\cite{evs}}
		If $f;X\longrightarrow Y$ ($X,Y$ being two exponential vector spaces over the same field $K$) is an order-morphism then $f(M):=\{f(m):m\in M\}$ is a subevs of $Y$ for any subevs $M$ of $X$.\label{p:rangesubevs}
	\end{Prop}
	\begin{Def}\cite{spri} In an evs $X$, the \textit{primitive of any $x\in X$} is defined as the set\\ \centerline{$P_x:=\{p\in X_0 : p\leq x\}$.} The axiom $A_6$ of the Definition \ref{d:evs} ensures that the primitive of each element of an evs is a nonempty set. Also, for any $x\in X$ and any scalar $\alpha$, $\alpha P_x=P_{\alpha x}$.
		\label{d:primitive}\end{Def}

	\begin{Def}\cite{spri} An evs $X$ is said to be a \emph{single primitive evs} if $P_x$ is a singleton set for each $x\in X.$ Also, in a single primitive evs $X$, $P_{x+y}=P_x+P_y$, $\forall\,x, y\in X$.
		
		Single primitivity is an evs property \cite{spri}.
		\label{d:sp}\end{Def}
	
	\begin{Def}
		An evs $X$ is said to be a \emph{zero primitive evs} if $P_x=\{\theta\}, \forall x\in X.$
		\label{d:zp}\end{Def}
	
	\begin{Def} \cite{JTh}
		An evs $X$ is said to be an \emph{additive primitive evs} if $P_{x+y}=P_x+P_y$,  $\forall x, y\in X.$ 
		
		Additive primitivity is an evs property \cite{JTh}.
		\label{d:ap}\end{Def}
	
	\begin{Def} \cite{bal}
		An element $x$ in an evs $X$ over $\K$ is said to be a \emph{balanced element} if $\alpha x\leq x, \forall \alpha\in \K$ with $|\alpha|\leq 1$.
		
		An evs $X$ over $\K$ is said to be a \emph{balanced evs} if each element of $X$ is a balanced element.	
		
		The property of an evs to be balanced is an evs property \cite{bal}.	
		\label{d:balanced}\end{Def}
		
	\begin{Def} \cite{JTh}
		An element $x$ in an evs $X$ over $\K$ is said to be a \emph{homogeneous element} if $\alpha x=|\alpha| x, \forall \alpha\in \K$. 
		
		An evs $X$ over $\K$ is said to be \emph{homogeneous} if each element of $X$ is a homogeneous element. A balanced evs is always homogeneous but not conversely.
		
		The property of an evs to be homogeneous is an evs property.
		\label{d:homo}\end{Def}
	
	\begin{Def} \cite{JTh}
		An element $x$ in an evs $X$ over $\K$ is said to be a \emph{convex element}\index{Convex element of an evs} if $(\alpha+\beta)x=\alpha x+\beta x, \forall \alpha, \beta \in\K$ with $\alpha, \beta \geq 0$.
		
		An evs $X$ over $\K$ is said to be \emph{convex}\index{Convex evs} if each element of $X$ is a convex element.
		
		The property of an evs to be convex is an evs property.
		\label{d:convex}\end{Def}
	
	\begin{Ex}{\cite{norm}}
	Let $\X$ be a vector space over the field $\K$ of real or complex numbers and $\mathcal{N}(\X)$ denotes the collection of all norms on $\X$ together with the zero function $O$. Define addition, scalar multiplication and partial order on $\mathcal{N}(\X)$ as follows:
	
	(i) For $f,g\in\mathcal{N}(\X)$, $(f+g)(x):= f(x) + g(x)$ for all $x\in\X$.
	
	(ii) For $f\in\mathcal{N}(\X)$ and for all $\alpha\in\K$, $(\alpha f)(x):= |\alpha| f(x)$ for all $x\in\X$.
	
	(iii) For $f,g\in\mathcal{N}(\X)$, $f\leq g\iff f(x)\leq g(x)$ for all $x\in\X$. \\
Then $\big(\mathcal{N}(\X),+,\cdot,\leq\big)$ becomes an evs over the field $\K$.

 $\norm$ is a single primitive,  convex, homogeneous and balanced  evs. Since the primitive space of $\mathcal{N}(\X)$ is $[\mathcal{N}(\X)]_{0}$ $=\{O\}$, it is a zero primitive evs. This evs becomes a topological evs if $\mathcal{N}(\X)$ is entitled with the topology of pointwise convergence. Any basic open set in $\norm$ is of the form  $\bigcap\limits_{i=1}^{n}W(x_i,U_i)$, where $ W(x_i,U_i):=\big\{f\in\norm:f(x_i)\in U_i\big\} $, $x_1,x_2,...,x_n\in\mathcal{X}$ and $U_1,U_2,...,U_n$ are open in  $[0,\infty)$.\label{e:spofnorms}
\end{Ex}

\section{Evs structure on the collection of metrics on a non-empty set}

	Let $\mathbf{X}$ be any non-empty set and $\mathcal{D}(\mathbf{X})$ denote the set of all metrics on $\mathbf{X}$ together with the constant zero function $O:X\times X\longrightarrow [0,\infty)$ defined by : $O(x,y)=0$, for all $x,y\in \mathbf{X}$. Clearly, $\mathcal{D}(\mathbf{X})\not=\emptyset$ since one can always find discrete metric on $\mathbf{X}$. Further $\mathcal{D}(\mathbf{X})$ is an uncountable set (since for any metric $\rho$ on $\mathbf{X}$ and for any two distinct positive real numbers $\alpha,\beta$, the metrics $\alpha \rho,\beta\rho$ are distinct).

	We define addition, scalar multiplication and partial order on $\mathcal{D}(\mathbf{X})$ as follows : 
	
	 for $\rho,\eta\in\mathcal{D}(\mathbf{X})$ and $\alpha\in\K$,
	
	(i) $(\rho+\eta)(x,y)$ := $\rho(x,y)+\eta(x,y)$, for all $x,y\in \mathbf{X}$. 
	
	(ii) $(\alpha \rho)(x,y):=|\alpha| \rho(x,y)$, for all $x,y\in \mathbf{X}$. 
	
	(iii) For any $\rho,\eta\in\mathcal{D}(\mathbf{X})$, $\rho\leq \eta$ $\Longleftrightarrow$ $\rho(x,y)\leq \eta(x,y)$, for all $x,y\in \mathbf{X}$. 
	
	Now we shall show that $\big(\mathcal{D}(\mathbf{X}), +, \cdot, \leq\big)$ forms an evs over the field $\K$.
	
	\begin{Th}
		$\mathcal{D}(\mathbf{X})$ is an exponential vector space over the field $\K$ of real or complex numbers with respect to the above mentioned operations and partial order.
	\end{Th}
	
	\begin{proof}
		Clearly addition of two metrics and scalar multiplication of a metric (as defined above) are again metrics. Also `$\leq$' is a partial order on $\mathcal{D}(\mathbf{X})$.
		
			Now we prove that $\big(\mathcal{D}(\mathbf{X}), +, \cdot, \leq\big)$ forms an evs over the field $\K$.
			
		$\bf{A_1:}$ Clearly, $(\mathcal{D}(\mathbf{X}), +)$ is a commutative semigroup with identity $O$, the constant function zero.
		
		$\bf{A_2:}$ Let $\rho,\eta\in\mathcal{D}(\mathbf{X})$ with $\rho\leq \eta$. Then $\rho(x,y)\leq \eta(x,y)$, for all $x,y\in \mathbf{X}$. For any $\zeta\in\mathcal{D}(\mathbf{X})$, $(\rho+\zeta)(x,y)=\rho(x,y)+\zeta(x,y)\leq \eta(x,y)+\zeta(x,y)=(\eta+\zeta)(x,y)$, for all $x,y\in \mathbf{X}$. This implies $\rho+\zeta\leq \eta+\zeta$, for all $\zeta\in\mathcal{D}(\mathbf{X})$.
		
		Also for any $\alpha\in\K$, $(\alpha \rho)(x,y)=|\alpha| \rho(x,y)\leq |\alpha| \eta(x,y)=(\alpha \eta)(x,y)$, for all $x,y\in \mathbf{X}$. This implies $\alpha \rho\leq \alpha \eta$, for all $\alpha\in\K$.
		
		$\bf{A_3\ (i):}$ Let $\rho,\eta\in\mathcal{D}(\mathbf{X})$ and $\alpha\in\K$. Then, $\big(\alpha(\rho+\eta)\big)(x,y)$ = $|\alpha|\big(\rho(x,y)+\eta(x,y)\big)$ = $|\alpha| \rho(x,y)+|\alpha| \eta(x,y)=(\alpha \rho+\alpha \eta)(x,y)$, for all $x,y\in \mathbf{X}$. This implies $\alpha(\rho+\eta)=\alpha \rho+\alpha \eta$, for all $\alpha\in\K$ and for all $\rho,\eta\in\mathcal{D}(\mathbf{X})$.
		
		$\bf{(ii)}$ Let $\rho\in\mathcal{D}(\mathbf{X})$ and $\alpha, \beta\in\K$. Then $\big(\alpha(\beta \rho)\big)(x,y)=|\alpha|(\beta \rho)(x,y)=|\alpha||\beta| \rho(x,y)$ = $|\alpha\beta| \rho(x,y)=\big((\alpha\beta)\rho\big)(x,y)$, for all $x,y\in \mathbf{X}$. This implies $\alpha(\beta \rho)=(\alpha\beta)\rho$, for all $\alpha,\beta\in\K$ and for all $\rho\in\mathcal{D}(\mathbf{X})$.
		
		$\bf{(iii)}$ Let $\rho\in\mathcal{D}(\mathbf{X})$ and $\alpha, \beta\in\K$. Then, $\big((\alpha+\beta)\rho\big)(x,y)$ = $|\alpha+\beta| \rho(x,y)$ $\leq$ $(|\alpha|+|\beta|)\rho(x,y)$ = $|\alpha| \rho(x,y)+|\beta| \rho(x,y)$ = $(\alpha \rho+\beta \rho)(x,y)$, for all $x,y\in \mathbf{X}$. This implies $(\alpha+\beta)\rho\leq \alpha \rho+\beta \rho$, for all $\alpha,\beta\in\K$ and for all $\rho\in\mathcal{D}(\mathbf{X})$.
		
		$\bf{(iv)}$ Clearly, $1\rho=\rho$, for all $\rho\in\mathcal{D}(\mathbf{X})$.
		
		$\bf{A_4:}$ Let $\alpha\in\K$ and $\rho\in\mathcal{D}(\mathbf{X})$. Then $\alpha \rho=O$ iff $(\alpha \rho)(x,y)=O(x,y)$, for all $x,y\in \mathbf{X}$ iff $|\alpha| \rho(x,y)=0$, for all $x,y\in \mathbf{X}$ iff $\alpha=0$ or $\rho=O$.
		
		$\bf{A_5:}$ For any $\rho\in\mathcal{D}(\mathbf{X})$, $\rho+(-1)\rho=O$ $\Longleftrightarrow$ $\big(\rho+(-1)\rho\big)(x,y)=0$, for all $x,y\in \mathbf{X}$ $\Longleftrightarrow$ $\rho(x,y)+\rho(x,y)=0$, for all $x,y\in \mathbf{X}$ $\Longleftrightarrow$  $\rho(x,y)=0$, for all $x,y\in \mathbf{X}$ i.e. $\rho=O$. Let $\rho\in\mathcal{D}(\mathbf{X})\smallsetminus\{O\}$. Then $\frac{1}{2}\rho(x,y)\leq\rho(x,y)$
		for all $x,y\in \mathbf{X}$ and $\frac{1}{2}\rho\not=\rho$ (since $\rho\not=O$). Thus $\rho\not\in[\mathcal{D}(\mathbf{X})]_{0}$. Therefore  we obtain  $[\mathcal{D}(\mathbf{X})]_0=\{O\}$. Thus $\rho+(-1)\rho=O$ $\Longleftrightarrow$ $\rho\in[\mathcal{D}(\mathbf{X})]_0$.
		
		$\bf{A_6:}$ Clearly, for each $\rho\in\mathcal{D}(\mathbf{X})$, $\rho(x,y)\geq 0$ for all $x,y\in \mathbf{X}$ i.e. $ \rho\geq O$ where, $O\in[\mathcal{D}(\mathbf{X})]_0$.
		
		Hence $\big(\mathcal{D}(\mathbf{X}),+,\cdot,\leq\big)$ forms an evs over $\K$.
		
	\end{proof}
	
	\begin{Th}
		$\mathcal{D}(\mathbf{X})$ is a single primitive, additive primitive, balanced, homogeneous and  convex evs.
		\label{t:D_pro}\end{Th}
	
	\begin{proof}
		As $[\mathcal{D}(\mathbf{X})]_0=\{O\}$, $\mathcal{D}(\mathbf{X})$ is zero primitive and hence single primitive and additive primitive evs. For $\alpha\in\K$ with $|\alpha|\leq1$, $(\alpha \rho)(x,y)=|\alpha| \rho(x,y)\leq \rho(x,y)$, for all $x,y\in \mathbf{X}$. This implies $\alpha \rho \leq \rho$ for $|\alpha|\leq1$. Thus $\mathcal{D}(\mathbf{X})$ is a balanced evs and hence a homogeneous evs. Now for $\alpha, \beta\in\K$ with $\alpha,\beta\geq0$, $(\alpha+\beta)\rho(x,y)=|\alpha+\beta| \rho(x,y)$ = $(|\alpha|+|\beta|)\rho(x,y)$ = $|\alpha| \rho(x,y)+|\beta| \rho(x,y)=(\alpha \rho+\beta \rho)(x,y)$, for all $x,y\in \mathbf{X}$. This implies $(\alpha+\beta)\rho=\alpha \rho+\beta \rho$ for $\alpha, \beta\geq0$. Hence, $\mathcal{D}(\mathbf{X})$ is a convex evs.
	\end{proof}

	We now provide the topology of pointwise convergence on $\mathcal{D}(\mathbf{X})\subseteq [0,\infty)^{\mathbf{X}\times\mathbf{X}}$, $[0,\infty)$ being endowed with the subspace topology inherited from $\R$ with usual topology.  For open set $U$ in $[0,\infty)$ and $(x,y)\in \mathbf X\times\mathbf X$, we denote $W\big((x,y),U\big):=\big\{\rho\in\mathcal{D}(\mathbf{X}):\rho(x,y)\in U\}$. Then the family of sets $\Big\{W\big((x,y),U\big):(x,y)\in\mathbf{X}\times\mathbf{X},U\text{ is open in }[0,\infty)\Big\}$ forms a subbase for the topology of pointwise convergence, say $\tau$, on $\mathcal{D}(\mathbf{X})$. In this topology, a net $\{\rho_{_\lambda}\}_{\lambda\in D} $ [$D$ being a directed set] in $\mathcal{D}(\mathbf{X})$ converges to some $\rho$ if and only if $\{\rho_{_\lambda}(x,y)\}_{\lambda\in D}$ converges to $\rho(x,y)$ in $[0,\infty)$, for each $(x,y)\in \mathbf{X}\times\mathbf{X}$.
	 		
\begin{Th}
Consider the evs $\mathcal{D}(\mathbf{X})$  with the topology of pointwise convergence. Then,		
(1) the addition `$+$' $:\mathcal{D}(\mathbf{X})\times\mathcal{D}(\mathbf{X})\longrightarrow\mathcal{D}(\mathbf{X})$ is continuous.\\
(2) the scalar multiplication `$\cdot$' $: \K\times\mathcal{D}(\mathbf{X})\longrightarrow\mathcal{D}(\mathbf{X})$ is continuous, where $\K$ is endowed with the usual topology.\\		
(3) the partial order `$\leq$' is closed.\\		
Thus $\mathcal{D}(\mathbf{X})$ with the topology of pointwise convergence forms a topological evs over the field $\K$.
\label{ch5:t:metric_top}\end{Th}
	
\begin{proof}
	(1)	Let $\{\rho_{_\lambda}\}_{\lambda\in D}$ and $\{\eta_{_\lambda}\}_{\lambda\in D}$ be two nets in $\mathcal{D}(\mathbf{X})$ ($D$ being a directed set) such that $\rho_{_\lambda}\to \rho$ and $\eta_{_\lambda}\to \eta$ in $\mathcal{D}(\mathbf{X})$. Then $\rho_{_\lambda}(x,y)\to \rho(x,y)$ and $\eta_{_\lambda}(x,y)\to \eta(x,y)$ in $[0,\infty)$, for all $x,y\in \mathbf{X}$. This implies $\rho_{_\lambda}(x,y)+\eta_{_\lambda}(x,y)\to \rho(x,y)+\eta(x,y)$, for all $ x,y\in \mathbf{X}$ \big(since $`+$' is continuous on $[0,\infty)$\big). Therefore $(\rho_{_\lambda}+\eta_{_\lambda})\to(\rho+\eta)$ in $\mathcal{D}(\mathbf{X})$. Hence the addition is continuous.
		
	(2)	Let $\{\rho_{_\lambda}\}_{\lambda\in D}$ be a net in $\mathcal{D}(\mathbf{X})$ and $\{\alpha_{_\lambda}\}_{\lambda\in D}$ be a net in $\K$ ($D$ being a directed set) such that $\rho_{_\lambda}\to \rho$ in $\mathcal{D}(\mathbf{X})$ and $\alpha_{_\lambda}\to \alpha$ in $\K$. Then $\rho_{_\lambda}(x,y)\to \rho(x,y)$ for all $x,y\in \mathbf{X}$. Now for any $x,y\in \mathbf{X}$,
		\begin{align*}
			&\big|(\alpha_{_\lambda} \rho_{_\lambda})(x,y)-(\alpha \rho)(x,y)\big|= \big||\alpha_{_\lambda}| \rho_{_\lambda}(x,y)-|\alpha| \rho(x,y)\big|\\
			&\hspace{4cm}=\big||\alpha_{_\lambda}| \rho_{_\lambda}(x,y)-|\alpha_{_\lambda}| \rho(x,y)+|\alpha_{_\lambda}| \rho(x,y)-|\alpha| \rho(x,y)\big|\\
			&\hspace{4cm} \leq|\alpha_{_\lambda}| \big| \rho_{_\lambda}(x,y)-\rho(x,y)\big| + | \rho(x,y)|\big||\alpha_{_\lambda}|-|\alpha|\big|\\
			&\hspace{4cm}\to0\quad  \big(\text{since } |\alpha_{_\lambda}|\to|\alpha|, \rho_{_\lambda}(x,y)\to \rho(x,y)\big)
		\end{align*}
		Therefore $ \alpha_{_\lambda} \rho_{_\lambda}\to\alpha \rho$ in $\mathcal{D}(\mathbf{X})$. Hence the scalar multiplication is continuous.
		
	(3)	Let $\{(\rho_{_\lambda}, \eta_{_\lambda})\}_{\lambda\in D}$ be a net in $\mathcal{D}(\mathbf{X})\times \mathcal{D}(\mathbf{X})$ ($D$ being a directed set) converging to some $(\rho,\eta)\in \mathcal{D}(\mathbf{X})\times \mathcal{D}(\mathbf{X})$ with $\rho_{_\lambda}\leq \eta_{_\lambda}$, for all $\lambda\in D$. So   $\rho_{_\lambda}\to \rho,\  \eta_{_\lambda}\to \eta$ in $\mathcal{D}(\mathbf{X})$. Therefore, $\rho_{_\lambda}(x,y)\to \rho(x,y)$, $\eta_{_\lambda}(x,y)\to \eta(x,y)$, for all $x,y\in \mathbf{X}$ and $\rho_{_\lambda}(x,y)\leq \eta_{_\lambda}(x,y)$, for all $x,y\in \mathbf{X}$, for all $\lambda\in D$. This implies $\rho(x,y)\leq \eta(x,y)$, for all $x,y\in \mathbf{X}$ \big(since partial order `$\leq$' is closed on the topological evs $[0, \infty)$\big). So $\rho\leq \eta$ in $\mathcal{D}(\mathbf{X})$.
		Hence the partial order $`\leq$' is closed on $\mathcal{D}(\mathbf{X})$.
		
		Therefore $\big(\mathcal{D}(\mathbf{X}),+,\cdot,\leq, \tau\big)$ is a topological evs over $\K$.
	\end{proof}
	
	 $\mathcal{D}(\mathbf{X})\subseteq$ $[0,\infty)^{\mathbf{X}\times \mathbf{X}}$, where $[0,\infty)$ is a Tychonoff space and $\tau$ is the topology of pointwise convergence. So $\big(\mathcal{D}(\mathbf{X}), \tau\big)$ is also a Tychonoff space. Hence this topology on $\mathcal{D}(\mathbf{X})$ is uniformizable. A base for this uniformity is given by\\ \centerline{$\mathscr{B}:=\Big\{V\big((x_1,y_1),\ldots,(x_n,y_n);\epsilon\big):(x_1,y_1),\ldots,(x_n,y_n)\in \mathbf{X}\times\mathbf{X},n\in\N,\epsilon>0\Big\}$ where,} $V\big((x_1,y_1),\ldots,(x_n,y_n);\epsilon\big)
	 :=\big\{(d,e)\in\mathcal{D}(\mathbf{X})\times\mathcal{D}(\mathbf{X}):|d(x_i,y_i)-e(x_i,y_i)|<\epsilon,i=1,\ldots,n\big\}$
	 
	 Now the natural question arises : is $\mathcal{D}(\mathbf{X})$ a complete uniform space ? The answer is `NO', in general, as is shown in the following example.
	 
	 	\begin{Ex}
	 	$\mathcal{D}(\R^2)$ is not a complete evs.\\	
\textbf{\underline{Justification} :} For each $n\in\N$, define $d_n:\R^2\times\R^2\longrightarrow[0,\infty)$ by \\ 	 
	 \centerline{$d_{n}(x,y):=|u-v|+\frac{1}{n}|u'-v'|,\forall\,x=(u,u'),y=(v,v')\in\R^2$}
	 
	 It is a routine work to verify that $d_n$ is a metric on $\R^2$ for each $n\in\N$. 
	 
	 Let $\epsilon>0$ be arbitrary. Then for any $n,m\in\N$ and $x_i=(u_i,u'_i),y_i=(v_i,v'_i)\in\R^2$, $i=1,\ldots,k$, $\big|d_{n}(x_i,y_i) -d_m(x_{i},y_{i})\big|=\Big|\big\{|u_{i}-v_{i}|+\frac{1}{n}|u'_{i}-v'_{i}|\big\}-\big\{|u_{i}-v_{i}|+\frac{1}{m}|u'_{i}-v'_{i}|\big\}\Big|=|u'_{i}-v'_{i}|\big|\frac{1}{n}-\frac{1}{m}\big|\leq C\big|\frac{1}{n}-\frac{1}{m}\big|$,  where $C:=\displaystyle\max_{1\leq i\leq k}|u'_{i}-v'_{i}|$. Now $\exists\,n_0\in\N$ such that $\big|\frac{1}{n}-\frac{1}{m}\big|<\frac{\epsilon}{C+1}$, $\forall\,n,m\geq n_0$. Hence $\big|d_{n}(x_i,y_i) -d_m(x_{i},y_{i})\big|<C\cdot\frac{\epsilon}{C+1}<\epsilon$, for $i=1,\ldots,k$, $\forall\,n,m\geq n_0$. Thus $(d_n,d_m)\in V\big((x_1,y_1),\ldots,(x_k,y_k);\epsilon\big),\forall\,n,m\geq n_0$. This implies $\{d_{n}\}_{n\in\N}$ is a Cauchy sequence in $\mathcal{D}(\R^2)$.
	 Now $d_{n}(x,y)\to|u-v|$ as $n\to\infty$, if $x=(u,u'),y=(v,v')\in\R^2$. If we define $d(x,y):=|u-v|$, whenever $x=(u,u'),y=(v,v')\in\R^2$ then $d$ is not a metric on $\R^2$, since $d(x,y)=0\notimplies x=y$. So $d\notin\mathcal{D}(\R^2)$. Therefore $\mathcal{D}(\R^2)$ is an incomplete uniform space.\qed
	  \end{Ex}
	  
We now show that $\mathcal{N}(\X)$, explained in Example \ref{e:spofnorms}, can be embedded homeomorphically and order-isomorphically into $\mathcal{D}(\X)$, whenever $\X$ is a linear space over the field $\K$.
	 
	\begin{Th}
		Let $\X$ be a linear space over the field $\K$. Define $\psi:\mathcal{N}(\mathcal{X})\to\mathcal{D}(\mathcal{X})$ by $\psi(f):=\widetilde{f}$,  for all $f\in\mathcal{N}(\mathcal{X})$, where $\widetilde{f}(x,y):=f(x-y)$, for all $x,y\in\mathcal{X}$. Then $\psi$ is an injective order-morphism.\label{t:injective}
		\end{Th}
	
		\begin{proof} Clearly $ \psi $ is well-defined. For any $f,g\in\norm$, $\psi(f+g)=\widetilde{f+g}=\widetilde{f}+\widetilde{g}=\psi(f)+\psi(g)$ \big[since $\widetilde{f+g}(x,y)=(f+g)(x-y)=f(x-y)+g(x-y)=\widetilde{f}(x,y)+\widetilde{g}(x,y)=(\widetilde{f}+\widetilde{g})(x,y)$, for all  $x,y\in\X$\big]. 
			
		For any $f\in\norm$ and $\alpha\in\K$, $\psi(\alpha f)=\widetilde{\alpha f}=\alpha \widetilde{f}=\alpha\psi(f)$ \big[since $\widetilde{\alpha f}(x,y)=(\alpha f)(x-y)$ = $|\alpha| f(x-y)$ = $|\alpha| \widetilde{f}(x,y)=(\alpha \widetilde{f})(x,y)$, for all  $x,y\in \X$ \big].
			
		For $f, g\in \norm$, $\psi(f)=\psi(g)\Longrightarrow \widetilde{f}=\widetilde{g} \Longrightarrow \widetilde{f}(x,y)=\widetilde{g}(x,y)$, $\forall\, x,y\in\X \Longrightarrow f(x-y)=g(x-y)$, $\forall\,x, y\in \X\Longrightarrow f(x)=g(x)$, $\forall\, x\in \X\Longrightarrow f=g$. So, $\psi$ is injective.
			
		Again, $f\leq g$ in $\norm$ $\Longleftrightarrow f(x)\leq g(x)$, $\forall x\in \X$ $\Longleftrightarrow f(x-y)\leq g(x-y), \forall\, x,y\in \X$ $\Longleftrightarrow \widetilde{f}(x,y)\leq \widetilde{g}(x,y)$, $\forall\, x,y\in\X$ $\Longleftrightarrow$ $\widetilde{f}\leq \widetilde{g}$ i.e. $\psi(f)\leq\psi(g)$ in $\mathcal{D}(\mathbf{X})$. This justifies that $ \psi $ is an order-morphism, since $\psi$ being injective for any $f\in \norm,\ \psi^{-1}\big(\psi(f)\big)=\{f\}$.
		\end{proof}
			
	\begin{Rem}
		By Proposition \ref{p:rangesubevs}, it follows that $\psi(\norm)$ is a subevs of $\mathcal{D(X)}$. Thus we can say that the collection of norms on a linear space $\X$ over the  field $\K$ can be embedded as a subevs of the collection of metrics on that linear space $\X$ over the same field $\K$.
	\end{Rem}
	
	\begin{Lem}
		The function $\psi:\mathcal{N(X)}\to\mathcal{D(X)}$ is  continuous.\label{l:normcont}
	\end{Lem}
	
	\begin{proof}
		Let $\{f_\lambda\}_{\lambda\in D}$ be a net in $\mathcal{N(X)}$ (where $D$ is a directed set) such that $\{f_\lambda\}$ converges to $f\in\mathcal{N(X)}$ in the topology of pointwise convergence. Hence $f_\lambda(x)\to f(x)$ for all $x\in\mathcal{X}$. Therefore for all $x,y\in\mathcal{X}$, $f_\lambda(x-y)\to f(x-y)$. This implies $\widetilde{f_\lambda}(x,y)\to \widetilde{f}(x,y)$ for all $x,y\in\X$. Hence $\psi(f_\lambda)=\widetilde{f_\lambda}$ converges to $\psi(f)=\widetilde{f}$ in the topology of pointwise convergence on $\mathcal{D(X)}$. Therefore $\psi$ is  continuous.
	\end{proof}
	
	\begin{Lem}
		$\psi:\mathcal{N(X)}\to\psi\big(\norm\big)$ is an open map.\label{l:normopen}
	\end{Lem}
	
	\begin{proof}
		Let $x_1,x_2,...,x_n\in\mathcal{X}$ and $U_1,U_2,...,U_n$ be open in  $[0,\infty)$. Any basic open set in $\norm$ is $\bigcap\limits_{i=1}^{n}W(x_i,U_i)$. Then		
\begin{align*}
&\hspace*{0.5cm}\psi\Big(\bigcap\limits_{i=1}^{n}W(x_i,U_i)\Big)\\ &=\bigcap\limits_{i=1}^n\psi\Big(W(x_i,U_i)\Big)\thinspace(\text{since }\psi\text{ is injective, by Theorem \ref{t:injective}})\\
&=\bigcap\limits_{i=1}^n\Big\{\psi(f):f\in\norm,f(x_i)\in U_i\Big\} \\
&=\bigcap\limits_{i=1}^n\Big\{\psi(f):\widetilde{f}(x_i,\theta)\in U_i\Big\}\\
&=\Big[\bigcap\limits_{i=1}^nW\big((x_i,\theta),U_i\big)\Big]\cap \psi\big(\norm\big)
\end{align*}		
Since $\bigcap\limits_{i=1}^{n} W\big((x_i,\theta),U_i\big)$ is a basic open set in $\mathcal{D(X)}$, $\psi\Big(\bigcap\limits_{i=1}^{n}W(x_i,U_i)\Big)$ is an open set in $\psi\big(\norm\big)$. Consequently $\psi$ is an open map. 
	\end{proof}
	
	\begin{Th}
		The map $\psi:\mathcal{N(X)}\to\psi\big(\mathcal{N(X)}\big)$ is a topological order-isomorphism. \label{t:normtopmor}
	\end{Th}
	
	\begin{proof}
		It directly follows from Theorem \ref{t:injective}, Lemma \ref{l:normcont} and Lemma \ref{l:normopen}.
	\end{proof}
	
	Thus the evs structure of $ \norm $ and $ \mathcal{D(\X)} $ are connected. We shall now compare two metrics or two norms in the following sections using the theory of exponential vector space.
	
\section{Comparability of metrics on a non-empty set}
	
Two metrics on a non-empty set can be compared in several ways. The most common is by comparing the topologies they induce. In this respect two metrics are called \textit{equivalent} if they produce the same topology. Also we may compare two metrics by comparing which topology, induced by them, is finer. Another one is by comparing their completeness property. In this section we shall compare two metrics in terms of `\textit{orderly dependence}', the \textit{order} being the partial order of the exponential vector space structure of $ \mathcal{D(\X)} $. This notion was introduced in the article `\textit{Basis and Dimension of Exponential Vector Space}' \cite{basis}. This mode of comparison also has some connection with the above mentioned `completeness' and `equivalence'. Let us first present some preliminary definitions and results related to basis and dimension of exponential vector space.
	
\begin{Def}\cite{basis}
		Let $X$ be an evs over $\K$ and $x\in X\smallsetminus X_0$. Define \\
\centerline{$L(x):=\big\{z\in X: z\geq \alpha x + p,\alpha\in\K^{*},p\in X_0\big\},\text{ where }\K^{*}\equiv \K\smallsetminus\{0\}.$}
		These sets $L(x)$ for different $x\in X\smallsetminus X_0$ are named as \textit{testing sets} of $X$.\index{Testing set}
	\end{Def}

	\begin{Prop}{\em\cite{basis}}
		(i) $\forall\, x\in X\smallsetminus X_{0}$, $x\in L(x)$ and $\uparrow L(x)=L(x)$,
		
		(ii) $x\leq y\,(x,y\in X\smallsetminus X_{0})\implies L(x)\supseteq L(y)$.
		
		(iii) If $x=\alpha y+p$ for some $\alpha\in\K^{*}$, $p\in X_{0}$ and $y\in X\smallsetminus X_{0}$, then $L(x)=L(y)$.
		
		(iv) $L(x)\cap X_{0}=\emptyset$.
		
		(v) If $a\in L(b)$ then $L(a)\subseteq L(b)$.
		
		(vi) For any $x,y\in X\smallsetminus X_{0}$, $L(x)\cap L(y)\not=\emptyset$. 
	\label{p:L(x)}\end{Prop}

	\begin{Def}\cite{basis}
		A subset $B$ of $X\smallsetminus X_0$ is said to be a \textit{generator} of $X\smallsetminus X_0$ if\\ \centerline{$\displaystyle X\smallsetminus X_0=\bigcup_{x\in B} L(x).$}
\label{d:gen}	\end{Def}

	\begin{Not}{\cite{basis}}
		The set $X\smallsetminus X_{0}$ always generates $X\smallsetminus X_{0}$. So generator always exists for $X\smallsetminus X_{0}$. Any superset of a generator of $X\smallsetminus X_{0}$ is also a generator of $X\smallsetminus X_{0}$.
	\end{Not}

	\begin{Def}{\cite{basis}}
		Two elements $x,y\in X\smallsetminus X_0$ are said to be \textit{orderly dependent} if either $x\in L(y)$ or $y\in L(x)$.
	\end{Def}

	\begin{Def}\cite{basis}
	Two elements $x,y\in X\smallsetminus X_0$ are said to be \textit{orderly independent} if they are not orderly
	dependent, i.e. neither $x\in L(y)$ nor $y\in L(x)$. 
	
	A subset $ B $ of $ X\smallsetminus X_0 $ is said to be \textit{orderly independent} if any two members $ x, y\in B $ are orderly	independent. Thus two orderly independent elements of $ X $ cannot be comparable with respect to the partial order.
\end{Def}

	\begin{Def}{\cite{basis}}
		A subset $B$ of $X\smallsetminus X_0$ is said to be a \textit{basis} of $X\smallsetminus X_0$ if $B$ is orderly independent and generates $X\smallsetminus X_0$. 
\label{d:basis}	\end{Def}

	\begin{Th}{\em\cite{basis}}
		For a topological evs $X$, $X\smallsetminus X_0$ either has no basis or has uncountably many bases.
	\end{Th}

	\begin{Th}{\em\cite{basis}}
		A set $B\,(\subset X\smallsetminus X_{0})$ is a basis of $X\smallsetminus X_{0}$ iff it is a minimal generating subset of $X\smallsetminus X_{0}$. {\em[}Here minimal generating subset $B$ of $X\smallsetminus X_{0}$ means there does not exist any proper subset of $B$ which can generate $X\smallsetminus X_{0}$.{\em]}
	\end{Th}

	\begin{Res}{\em\cite{basis}}
		Every basis of $X\smallsetminus X_{0}$ is a maximal orderly independent subset of $X\smallsetminus X_{0}$. {\em[}Here maximal orderly independent subset $B$ of $X\smallsetminus X_{0}$ means there does not exist any orderly independent subset of $X\smallsetminus X_{0}$ containing $B$ properly.{\em]}
	\end{Res}

		Converse of this result is false in general i.e any maximal orderly independent subset of $X\smallsetminus X_{0}$ may fail to be a basis of $X\smallsetminus X_{0}$.
		
	\begin{Th}{\em\cite{basis}}
		If $A$ and $B$ are two bases of $X\smallsetminus X_{0}$, then $\text{card }A=\text{card }B$.
	\end{Th}

	\begin{Def}\cite{basis}
		For an evs $X$, we define \textit{dimension} of $X\smallsetminus X_0$ as\\
		\centerline{$\text{dim }(X\smallsetminus X_0):= \text{card}(B),\text{ where } B\text{ is a basis of }X\smallsetminus X_0$.} Then we shall represent dimension of the evs $X$ as $\text{ dim }(X):=[\text{dim}(X\smallsetminus X_0):\text{dim }X_0]$. If $X_0=\{\theta\}$, then dimension of $X_0$ will be taken as $0$, since then $X_0$ has no basis [as vector space].
	\end{Def}
	
	\begin{Res}{\em\cite{basis}}
		Let $X$ be an evs and $B$ be a basis of $X\smallsetminus X_{0}$. Then $\downarrow x\smallsetminus X_{0}\subseteq L(x)$, for each $x\in B$.\label{r:basefeas}
	\end{Res}

	\begin{Def}{\cite{basis}}
		For an evs $X$, let $Q(X):=\big\{x\in X\smallsetminus X_{0}:(\downarrow x\smallsetminus X_{0})\subseteq L(x)\big\}$. The set $Q(X)$ is called \textit{feasible set} of $X$.\label{d:feasible}
	\end{Def}

	\begin{Rem}{\cite{JTh}}
		From Result \ref{r:basefeas}, it follows that if there exists a basis $B$ of $X\smallsetminus X_{0}$, then $B\subseteq Q(X)$. If for an evs $ X $, $Q(X)=\emptyset$, then $X\smallsetminus X_{0}$ does not have a basis. Also for an evs $X$, $Q(X)\not=\emptyset$ does not necessarily imply that $X\smallsetminus X_{0}$ possesses a basis.
\label{rm:basefeas}\end{Rem}

	\begin{Lem}{\em\cite{basis}}
		For an evs $X$, if $x\in Q(X)$ then for each $y\in\downarrow x\smallsetminus X_{0}$, $L(x)=L(y)$.
\label{l:Q(X)}	\end{Lem}

	\begin{Th}{\em\cite{basis}}
		An evs $X$ has a basis iff $Q(X)$ is a generator of $X\smallsetminus X_{0}$.\label{t:Q(X)generator}
	\end{Th}

	\begin{Th}{\em\cite{basis}}
		For an evs $X$, every maximal orderly independent set of $Q(X)$ is a basis of $X\smallsetminus X_{0}$, provided $Q(X)$ generates $X\smallsetminus X_{0}$.
\label{t:maxbasis}\end{Th}
	
	Now we want to discuss the comparability of two metrics on a nonempty set $ \mathbf X $ in terms of orderly dependence. This necessitates finding of basis of $\mathcal{D}(\mathbf{X})\smallsetminus\{O\}$. So it becomes indispensable to determine the feasible set $ Q\big(\mathcal D(\mathbf X)\big) $.
	
	Let us use the notation $\rho_{_d}$ for the discrete metric on $\mathbf{X}$. Further we use symbol $\rho_{_u}$ for denoting usual metric on $\R$. We shall show that these two metrics on $\R$ are orderly independent.
	
	\begin{Res}
		Let $\rho\in\mathcal{D}(\mathbf{X})\smallsetminus\{O\}$ be any bounded metric. Define $\rho_{_b}:\mathbf{X}\times\mathbf{X}\longrightarrow[0,\infty)$ by $\rho_{_b}(x,y):=\frac{\rho(x,y)}{1+\rho(x,y)}$ for all $x,y\in\mathbf{X}$. Then $\rho\in L(\rho_{_b})$ and $\rho_{_b}\in L(\rho)$.
	\end{Res}

	\begin{proof}
		Clearly, $\rho_{_b}\not=O$. Now $\rho_{_b}(x,y)=\frac{\rho(x,y)}{1+\rho(x,y)}\leq \rho(x,y)$ for all $x,y\in\mathbf{X}$. So $\rho_{_b}\leq \rho$. Hence $\rho\in L(\rho_{_b})$.
		
		Since $\rho$ is a bounded metric, there exists $M>1$ such that $\rho(x,y)\leq M-1$ for all $x,y\in\mathbf{X}$. So $ \frac{1}{1+ \rho(x,y)}\geq \frac{1}{M}$ for all $x,y\in\mathbf{X}$. Therefore $\frac{1}{M}$ $\rho(x,y)\leq\frac{\rho(x,y)}{1+\rho(x,y)}=\rho_{_b}(x,y)$ for all $x,y\in\mathbf{X}$ and hence $\frac{1}{M}\cdot \rho\leq \rho_{_b}$. Thus $\rho_{_b}\in L(\rho)$.
	\end{proof}

\begin{Rem}
	For any metric $\rho$ on $\mathbf{X}$, the metric $\rho_{_b}$ is orderly dependent with the metric $\rho$.
\end{Rem}

	\begin{Th}
		Let $\rho\in\mathcal{D}(\mathbf{X})\smallsetminus\{O\}$ be any bounded metric. Define $\rho_{_{\min}}:\mathbf{X}\times\mathbf{X}\longrightarrow[0,\infty)$ by $\rho_{_{\min}}(x,y):=\min\{1,\rho(x,y)\}$ for all $x,y\in\mathbf{X}$. Then $\rho_{_{\min}}\in L(\rho)$ and $\rho\in L(\rho_{_{\min}})$.
\label{t:min}	\end{Th}

	\begin{proof}
		Clearly $\rho_{_{\min}}\not=O$. Then $\rho_{_{\min}}(x,y)=\min\{1,\rho(x,y)\}\leq\rho(x,y)$ for all $x,y\in\mathbf{X}$ i.e. $\rho_{_{\min}}\leq \rho$. Hence $\rho\in L(\rho_{_{\min}})$. 
		
		Since $\rho$ is bounded, there exists $0<M\leq 1$ such that $\rho(x,y)\leq\frac{1}{M}$ for all $x,y\in\mathbf{X}$. Therefore $M\rho(x,y)\leq 1$ for all $x,y\in\mathbf{X}$ and $M\rho(x,y)\leq\rho(x,y)$ for all $x,y\in\mathbf{X}$ (since $M\leq 1$). So $M\rho(x,y)\leq\min\{1,\rho(x,y)\}=\rho_{_{\min}}(x,y)$ for all $x,y\in\mathbf{X}$. Hence $M\cdot\rho\leq\rho_{_{\min}}$, where $M>0$. This implies $\rho_{_{\min}}\in L(\rho)$. 
	\end{proof}

	\begin{Res}
		The discrete metric $\rho_{_d}$ and usual metric $\rho_{_u}$ on $\R$ are orderly independent.
\label{r:usudis}\end{Res} 

	\begin{proof}
		If possible let $\rho_{_u}\in L(\rho_{_d})$. Then there exists $\alpha\in\K^{*}$ such that $\rho_{_u}\geq\alpha\cdot \rho_{_d} + O$. Therefore $\rho_{_u}(x,y)\geq |\alpha|\rho_{_d}(x,y)$ for all $x,y\in\R$. Hence $\displaystyle\inf_{\underset{x\not=y}{x,y\in\R}} \rho_{_u}(x,y)\geq|\alpha|>0$, which is a contradiction. Therefore $\rho_{_u}\not\in L(\rho_{_d})$.
		
		Now if $\rho_{_d}\in L(\rho_{_u})$, then there exists $\beta\in\K^{*}$ such that $\rho_{_d}\geq \beta\cdot \rho_{_u}+O$. So $ \rho_{_d}(x,y)\geq |\beta|\rho_{_u}(x,y)$ for all $x,y\in\R$. Therefore $ \rho_{_u}(x,y)\leq\frac{1}{|\beta|}$ for all $x,y\in\R$, which is a contradiction since $\rho_{_u}$ is an unbounded metric. Hence $\rho_{_d}\not\in L(\rho_{_u})$.
		
		Therefore $\rho_{_d}$ and $\rho_{_u}$ are orderly independent.
	\end{proof}
	
	\begin{Rem}
		If there exists an unbounded metric $d$ on $\mathbf{X}$ such that $\displaystyle\inf_{\underset{x\not=y}{x,y\in\mathbf{X}}} d(x,y)=0$ and a bounded metric $\rho$ such that $\displaystyle\inf_{\underset{x\not=y}{x,y\in\mathbf{X}}} \rho(x,y)>0$, then $d$ and $\rho$ are orderly independent. 
	\end{Rem}

We now discuss the structure of $ Q\big(\mathcal D(\mathbf X)\big) $. Although we are not able to find the exact structure of $ Q\big(\mathcal D(\mathbf X)\big) $ we can find some clue that is helpful to determine a basis of $ \mathcal{ D(\mathbf X)}\smallsetminus\{O\} $. Then using the structure of basis we can compare two metrics in terms of orderly dependence.

	\begin{Lem}
		Let $A:=\big\{\rho\in\mathcal{D}(\mathbf{X}):\rho\text{ is an unbounded metric}\big\}$. Then $A\cap Q\big(\mathcal{D}(\mathbf{X})\big)=\emptyset$.\label{l:unbdnotfeas}
	\end{Lem}

	\begin{proof} Let $\rho\in A$. We claim that $\downarrow \rho\smallsetminus\{O\}\not\subseteq L(\rho)$. Consider the metric $ \rho_{_b}(x,y):=\frac{\rho(x,y)}{1+\rho(x,y)} $ for all $x,y\in\mathbf{X}$. 
		Clearly $\rho_{_b}\not=O$ and $\rho_{_b}(x,y)\leq \rho(x,y)$ for all $x,y\in\mathbf{X}$ which implies $\rho_{_b}\leq \rho$ and hence $\rho_{_b}\in\downarrow \rho\smallsetminus\{O\}$. We now show below that $ \rho_{_b}\notin L(\rho) $.
		
		If possible, let $\rho_{_b}\in L(\rho)$. Then $\exists\,\alpha\in\K^{*}$ such that $\rho_{_b}\geq \alpha\cdot \rho+O$. So $\rho_{_b}(x,y)\geq |\alpha|\rho(x,y)$, $\forall\,x,y\in\mathbf{X}$. Again $\rho_{_b}(x,y)=\frac{\rho(x,y)}{1+\rho(x,y)}\leq 1$, $\forall\,x,y\in\mathbf{X}$. Thus $\rho(x,y)\leq\frac{1}{|\alpha|}$, $\forall\,x,y\in\mathbf{X}$, which contradicts that $\rho$ is unbounded. Hence $\rho_{_b}\not\in L(\rho)$.
		
		Therefore $\downarrow \rho\smallsetminus\{O\}\not\subseteq L(\rho)$ and hence $\rho\not\in Q\big(\mathcal{D}(\mathbf{X})\big)$ (see Definition \ref{d:feasible}). This implies $A\cap Q\big(\mathcal{D}(\mathbf{X})\big)=\emptyset$.
	\end{proof}

		\begin{Th}
		If $\mathcal{D}(\mathbf{X})\smallsetminus\{O\}$ has a basis $B$, then $B$ does not contain any unbounded metric.
	\end{Th}

	\begin{proof}
		It follows from Lemma \ref{l:unbdnotfeas} and Remark \ref{rm:basefeas}.
	\end{proof}
	
	\begin{Th}
		$\mathcal{D}(\mathbf{X})\smallsetminus\{O\}$ has a basis, if there does not exist any metric $d\in\mathcal{D}(\mathbf{X})$ such that $\displaystyle\inf_{\underset{x\not=y}{x,y\in\mathbf{X}}} d(x,y)=0$.\end{Th}
	
	\begin{proof}
		Let $\rho\in\mathcal{D}(\mathbf{X})\smallsetminus\{O\}$. Then $\displaystyle\inf_{\underset{x\not=y}{x,y\in\mathbf{X}}} \rho(x,y)=:a>0$ (by hypothesis). Thus $a\leq \rho(x,y)$ for all $x,y\in\mathbf{X}$ with $x\not=y$. Therefore $a \rho_{_d}(x,y)\leq \rho(x,y)$ for all $x,y\in\mathbf{X}$. So $a\cdot\rho_{_d} + O\leq \rho$ which implies $\rho\in L(\rho_{_d})$. Hence $ B=\{\rho_{_d}\}$ is a basis of $\mathcal{D}(\mathbf{X})\smallsetminus\{O\}$. 
		
		The primitive space of $\mathcal{D}(\mathbf{X})$ is $[\mathcal{D}(\mathbf{X})]_{0}=\{O\}$. Hence the primitive space does not have a basis. 
				
		 Therefore $\text{dim }\mathcal{D}(\mathbf{X})=\big[\text{dim }\big(\mathcal{D}(\mathbf{X})\smallsetminus[\mathcal{D}(\mathbf{X})]_{0}\big):\text{dim }[\mathcal{D}(\mathbf{X})]_{0}\big]=[1:0]$.		\end{proof}
		
We now discuss the case when there is a metric $d\in\mathcal{D}(\mathbf{X})$ such that $\displaystyle\inf_{\underset{x\not=y}{x,y\in\mathbf{X}}} d(x,y)=0$. Such metric will be called `\textit{shrinking}'. If $ \mathbf X=\R $ then the usual metric is shrinking whereas the discrete metric is not shrinking. Thus from above theorem we can say that if there is no shrinking metric on $ \mathbf{X} $ then $ \mathcal{D}(\mathbf X)\smallsetminus\{O\} $ possesses a basis, namely $ \{\rho_{_d}\} $, which means that any metric on such $ \mathbf{X} $ can be compared with the discrete metric. Such $ \mathbf X $ exists; in fact, there is no shrinking metric on any finite set $ \mathbf{X} $. On the other hand, if $ \mathbf X $ is infinite then we can construct a metric which is shrinking, as the following example shows.

\begin{Ex}
	Let $ \mathbf X $ be an infinite set and $ A:=\{x_n\in\mathbf X:n\in\N\} $ be a countably infinite subset of $ \mathbf X $. Define a metric on $ \mathbf X $ as follows:\\
	$ d(x,y)=\begin{cases}
		|\frac{1}{n}-\frac{1}{m}|,\text{ if }x=x_n,y=x_m\in A\\
		1,\text{ if }x\neq y\text{ and at least one of }x,y\text{ does not belong to }A\\
		0,\text{ otherwise}
	\end{cases} $\\
This metric $ d $ is shrinking since $\displaystyle\inf_{\underset{x\not=y}{x,y\in\mathbf{X}}} d(x,y)=0$. \qed
\end{Ex}

Thus our finding aims now to discuss the existence of basis of $\mathcal{D}(\mathbf{X})\smallsetminus\{O\}$ when $ \mathbf X $ is an infinite set. The discussion through the following Theorems and Lemma will show the existence of a generator of $\mathcal{D}(\mathbf{X})\smallsetminus\{O\}$ considering the presence of a shrinking metric on $ \mathbf X $ (Theorem \ref{t:generator}); this will finally culminate into the existence of a basis of $\mathcal{D}(\mathbf{X})\smallsetminus\{O\}$ under some condition (Theorem \ref{t:basis}).

	\begin{Th}
		Suppose there is a metric $d\in\mathcal{D}(\mathbf{X})$ such that $\displaystyle\inf_{\underset{x\not=y}{x,y\in\mathbf{X}}}d(x,y)=0$. If $\mathcal{D}(\mathbf{X})\smallsetminus\{O\}$ possesses a basis $B$, then the discrete metric $\rho_{_d}\not \in B$. 
	\end{Th}

	\begin{proof}		
		Consider the metric $\widehat{\rho_{_b}}$ defined by $\widehat{\rho_{_b}}(x,y)=:\frac{d(x,y)}{1+d(x,y)}$ for all $x,y\in\mathbf{X}$. It satisfies the following condition: $\widehat{\rho_{_b}}(x,y)\leq d(x,y)$ for all $x,y\in\mathbf{X}$. Hence
		$\widehat{\rho_{_b}}\leq d\text{ but }\widehat{\rho_{_b}}\not=O$.	 Also $\widehat{\rho_{_b}}(x,y)\leq 1$ for all $x,y\in\mathbf{X}$. Then $\widehat{\rho_{_b}}\leq \rho_{_d}$ i.e. $\widehat{\rho_{_b}}\in \downarrow \rho_{_d}\smallsetminus\{O\}$. If $\widehat{\rho_{_b}}\in L(\rho_{_d})$, there exists $\beta\in\K^{*}$ such that $\widehat{\rho_{_b}}\geq\beta\cdot\rho_{_d} + O$ which implies $\widehat{\rho_{_b}}(x,y)\geq |\beta|\rho_{_d}(x,y)=|\beta|$ for all $x,y\in\mathbf{X}$ with $x\not=y$. Therefore $\displaystyle\inf_{\underset{x\not=y}{x,y\in\mathbf{X}}} \widehat{\rho_{_b}}(x,y)\geq |\beta|>0$, which is a contradiction because $\displaystyle\inf_{\underset{x\not=y}{x,y\in\mathbf{X}}} \widehat{\rho_{_b}}(x,y)\leq \displaystyle\inf_{\underset{x\not=y}{x,y\in\mathbf{X}}} d(x,y)=0$. This implies $\widehat{\rho_{_b}}\not\in L(\rho_{_d})$. Therefore $\downarrow\rho_{_d}\smallsetminus\{O\}\not\subseteq L(\rho_{_d})$. Hence by Result \ref{r:basefeas}, $\rho_{_d}\not\in B$.   	
	\end{proof}

	\begin{Lem}
			Suppose there is a metric $d\in\mathcal{D}(\mathbf{X})$ such that $\displaystyle\inf_{\underset{x\not=y}{x,y\in\mathbf{X}}}d(x,y)=0$. Define\\ $E:=\big\{\rho\in \mathcal{D}(\mathbf{X}): \rho \text{ is bounded and }\displaystyle\inf_{\underset{x\not=y}{x,y\in\mathbf{X}}} \rho(x,y)>0\big\}$. Then $ E\cap Q\big(\mathcal{D}(\mathbf{X})\big)=\emptyset$.\label{l:bdmetinfpos}
	\end{Lem}

	\begin{proof}
			Consider the metric $\widehat{\rho_{_b}}$ defined by $\widehat{\rho_{_b}}(x,y):=\frac{d(x,y)}{1+d(x,y)}$ for all $x,y\in\mathbf{X}$.
		Clearly $\widehat{\rho_{_b}}\leq d$. Let $\rho\in E$. Denote $c:=\displaystyle\inf_{\underset{x\not=y}{x,y\in\mathbf{X}}}\rho(x,y)$. Then $\frac{1}{c}\cdot \rho\in E$ (since $c>0$). Also $\widehat{\rho_{_b}}(x,y)\leq 1\leq \frac{1}{c}\cdot\rho(x,y)$, $\forall\,x,y\in\mathbf{X}$. This implies $c\cdot \widehat{\rho_{_b}}\leq \rho$ and $\widehat{\rho_{_b}}\not=O$ i.e. $c \cdot \widehat{\rho_{_b}}\in \downarrow \rho\smallsetminus\{O\}$. We claim that $c\cdot \widehat{\rho_{_b}}\not\in L(\rho)$. 
		
		If not, let $c\cdot \widehat{\rho_{_b}}\in L(\rho)$. Then  $\exists\,\alpha\in\K^{*}$ such that $c$ $\widehat{\rho_{_b}}(x,y)\geq |\alpha|\rho(x,y)$, $\forall\,x,y\in\mathbf{X}$ which  implies $\widehat{\rho_{_b}}(x,y)\geq \frac{|\alpha|\rho(x,y)}{c}\geq |\alpha|$, $\forall\,x,y\in\mathbf{X}$ with $x\not=y$. Hence $\displaystyle\inf_{\underset{x\not=y}{x,y\in\mathbf{X}}} \widehat{\rho_{_b}}(x,y)\geq|\alpha|>0$ ------ a contradiction, since $\displaystyle\inf_{\underset{x\not=y}{x,y\in\mathbf{X}}} \widehat{\rho_{_b}}(x,y)\leq \displaystyle\inf_{\underset{x\not=y}{x,y\in\mathbf{X}}} d(x,y)=0$. Therefore $c\cdot \widehat{\rho_{_b}}\not\in L(\rho)$. 
		
		Thus we have $\downarrow\rho\smallsetminus\{O\}\not\subseteq L(\rho)$ and consequently, $\rho\not\in  Q\big(\mathcal{D}(\mathbf{X})\big)$ (see Definition \ref{d:feasible}). Therefore $E\cap Q\big(\mathcal{D}(\mathbf{X})\big)=\emptyset$.
	\end{proof}

	\begin{Th}
		Suppose there exists a metric $d\in\mathcal{D}(\mathbf{X})$ such that $\displaystyle\inf_{\underset{x\not=y}{x,y\in\mathbf{X}}}d(x,y)=0$. If $\mathcal{D}(\mathbf{X})\smallsetminus\{O\}$ has a basis $B$, then $B\cap E=\emptyset$. {\em[}Here $ E $ is as in Lemma \ref{l:bdmetinfpos}{\em]}
	\end{Th}

	\begin{proof}
		Proof of this Theorem follows from above Lemma \ref{l:bdmetinfpos} and Remark \ref{rm:basefeas}.
	\end{proof}
	
	\begin{Th} Suppose there exists a metric $d\in\mathcal{D}(\mathbf{X})$ such that $\displaystyle\inf_{\underset{x\not=y}{x,y\in\mathbf{X}}}d(x,y)=0$.
		Then the set $D:=\big\{\rho\in\mathcal{D}(\mathbf{X}):\rho\text{ is a bounded metric with }\displaystyle\inf_{\underset{x\not=y}{x,y\in\mathbf{X}}} \rho(x,y)=0\big\}$ is a generator of $\mathcal{D}(\mathbf{X})\smallsetminus\{O\}$. \label{t:generator}
	\end{Th}

	\begin{proof} Define a metric $ \widehat{d}(x,y):=\frac{d(x,y)}{1+d(x,y)},\forall\,x,y\in\mathbf X $. Then $ \widehat{d} $ is bounded with $ \displaystyle\inf_{\underset{x\not=y}{x,y\in\mathbf{X}}}\widehat{d}(x,y)=0 $. So $ \widehat{d} \in D$ i.e. $ D\neq\emptyset $. Now let $\rho\in \mathcal{D}(\mathbf{X})\smallsetminus\{O\}$ such that $\displaystyle\inf_{\underset{x\not=y}{x,y\in\mathbf{X}}} \rho(x,y)=:s>0$. Then
$ \widehat{d}(x,y)\leq\frac{1}{s}\cdot s\leq\frac{1}{s}\rho(x,y),\forall\,x,y\in\mathbf X,x\neq y $. So $ s\cdot\widehat d\leq\rho $ and hence $ \rho\in L(\widehat{d}) $, where $ \widehat{d}\in D $.

Next let $\rho\in \mathcal{D}(\mathbf{X})\smallsetminus\{O\}$ such that $\displaystyle\inf_{\underset{x\not=y}{x,y\in\mathbf{X}}} \rho(x,y)=0$. Define the metric $ \widehat{\rho}(x,y):=\frac{\rho(x,y)}{1+\rho(x,y)},\forall\,x,y\in\mathbf X $. Then $ \widehat{\rho} $ is bounded with $ \displaystyle\inf_{\underset{x\not=y}{x,y\in\mathbf{X}}} \widehat{\rho}(x,y)=0 $. So $ \widehat{\rho} \in D$. Clearly $ \widehat{\rho}\leq\rho $ which implies $ \rho\in L(\widehat{\rho}) $. Thus $ D $ is a generator of $\mathcal{D}(\mathbf{X})\smallsetminus\{O\}$.	\end{proof}

\begin{Not} 
If there exists a metric $d\in\mathcal{D}(\mathbf{X})$ such that $\displaystyle\inf_{\underset{x\not=y}{x,y\in\mathbf{X}}}d(x,y)=0$ then in view of Lemma \ref{l:unbdnotfeas} and Lemma \ref{l:bdmetinfpos}, we can say that\\ $$Q\big(\mathcal{D}(\mathbf{X})\big)\subseteq D:=\left\{\rho\in\mathcal{D}(\mathbf{X}):\rho\text{ is a bounded metric with }\displaystyle\inf_{\underset{x\not=y}{x,y\in\mathbf{X}}}\rho(x,y)=0\right\}$$.
\end{Not}
	
\noindent If $Q\big(\mathcal{D}(\mathbf{X})\big) = D$, then we can say that $\mathcal{D}(\mathbf{X})\smallsetminus\{O\}$ has a basis, in view of Theorems \ref{t:Q(X)generator} and \ref{t:generator}. So the question is whether $Q\big(\mathcal{D}(\mathbf{X})\big) = D$ ? The equality may not be true in general. We can construct, as shown below, a metric on $[-1,1]$ which lies outside $Q\big(\mathcal{D}([-1,1])\big)$ but inside $D$ i.e. $Q\big(\mathcal{D}([-1,1])\big)\subsetneqq D$. 

\begin{Ex}

		Define $\kappa:[-1,1]\times[-1,1]\longrightarrow[0,\infty)$ by \\		
			\centerline{$ \kappa(x,y):= \begin{cases} 			
				|x-y|, & \text{if } |x|\leq \frac{1}{2}\text{ and }|y|\leq \frac{1}{2}\\
				2 ,& \text{if } \frac{1}{2}<|x|\leq 1\text{ or }\frac{1}{2}<|y|\leq 1\text{ with }x\not=y \\
				0,& \text{otherwise}
			\end{cases} $}
	Now we shall show that $\kappa$ is a metric on $[-1,1]$. 
	
	(i) Clearly $\kappa(x,y)\geq 0$ for all $x,y\in[-1,1]$. 
	
	(ii)It is easy to observe that $\kappa(x,y)=0\iff x=y$.
	
	(iii) If $x,y\in[-\frac{1}{2},\frac{1}{2}]$, then $\kappa(x,y)=|x-y|=|y-x|=\kappa(y,x)$. If $x,y\in [-1,1]$ with $|x|>\frac{1}{2}$ or $|y|>\frac{1}{2}$, then for $x\not=y$, $\kappa(x,y)=2=\kappa(y,x)$ and for $x=y$, $\kappa(x,y)=0=\kappa(y,x)$. Therefore $\kappa(x,y)=\kappa(y,x)$ for all $x,y\in[-1,1]$.
	
	(iv) Let $x,y\in[-1,1]$ and $z\in[-1,1]$. We want to verify triangle inequality for $\kappa$.
	
	\textbf{Case 1:} Suppose $|x|\leq \frac{1}{2}$ and $|y|\leq \frac{1}{2}$. Then for any $z\in [-1,1]$ with $|z|\leq \frac{1}{2}$, $\kappa(x,z)=|x-z|$ and $\kappa(y,z)=|y-z|$. Hence $\kappa(x,y)=|x-y|\leq |x-z|+|y-z|=\kappa(x,z)+\kappa(z,y)$. If $z\in[-1,1]$ with $|z|>\frac{1}{2}$, we obtain $\kappa(x,z)=2$ and $\kappa(z,y)=2$. Hence $\kappa(x,y)=|x-y|\leq 1\leq \kappa(x,z)+\kappa(z,y)$. 
	
	\textbf{Case 2:} Suppose $|x|>\frac{1}{2}$ and $|y|\leq\frac{1}{2}$. For any $z\in [-1,1]$ with $z\not=x$, we observe $\kappa(x,z)=2$. Hence $\kappa(x,y)= 2\leq 2+\kappa(z,y)=\kappa(x,z)+\kappa(z,y)$. If $z=x$, then $\kappa(z,x)=0$ and consequently, we obtain $\kappa(x,y)=\kappa(x,z)+\kappa(z,y)$ [since $\kappa(x,y)=\kappa(z,y)$].
	
	\textbf{Case 3:} Suppose $|x|\leq \frac{1}{2}$ and $|y|>\frac{1}{2}$. Then for any $z\in[-1,1]$, $\kappa(x,y)\leq \kappa(x,z)+\kappa(z,y)$ (similarly by Case 2).
	
	\textbf{Case 4:} Suppose $|x|>\frac{1}{2}$ and $|y|>\frac{1}{2}$. Then for any $z\in[-1,1]$ with $z\not=x$, we have $\kappa(x,z)=2$ and $\kappa(z,y)=2$. Therefore $\kappa(x,y)=2\leq 2+2= \kappa(x,z)+\kappa(z,y)$. If $z=x$, then $\kappa(x,z)=0$ and so, $\kappa(x,y)=\kappa(z,y)$. As a consequence, $\kappa(x,y) = \kappa(x,z)+\kappa(z,y)$.
	 
	Hence $\kappa$ is a metric on $[-1,1]$.\qed
\end{Ex}

	\begin{Res}
		$\kappa\in D\smallsetminus Q\Big(\mathcal{D}\big([-1,1]\big)\Big)$.
	\end{Res} 

	 \begin{proof}
	 	Since $\kappa(x,y)\leq 2$ for all $x,y\in[-1,1]$, $\kappa$ is a bounded metric. Also for $x,y\in[-\frac{1}{2},\frac{1}{2}]$, $\kappa(x,y)=|x-y|$. Therefore $ \inf\big\{\kappa(x,y):x,y\in[-\frac{1}{2},\frac{1}{2}]\text{ and }x\not=y\big\} =0$. Also $\inf\big\{\kappa(x,y):1\geq|x|>\frac{1}{2} \text{ or } 1\geq|y|>\frac{1}{2}, \, x\not=y\big\}=2$. Therefore we  obtain that $\displaystyle\inf_{\underset{x\not=y}{x,y\in[-1,1]}} \kappa(x,y)=0$. Hence $\kappa\in D$. 
	 	
	 	Now we prove that $\kappa\not \in Q\Big(\mathcal{D}\big([-1,1]\big)\Big)$. Clearly, $\rho_{_u}(x,y)= \kappa(x,y)$ for all $x,y\in [-\frac{1}{2},\frac{1}{2}]$ and $\rho_{_u}(x,y)\leq \kappa(x,y)$ for all other $x,y\in [-1,1]$. Also $\rho_{_u}\not={O}$. Hence $\rho_{_u}\in \downarrow \kappa\smallsetminus\{O\}$. 
	 	
	 	We claim that $\rho_{_u}\notin L(\kappa)$. If not, let $\rho_{_u}\in L(\kappa)$. Then $\exists\,\alpha\in\K^{*}$ such that $\rho_{_u}\geq \alpha\cdot \kappa$ i.e. $\rho_{_u}(x,y)\geq |\alpha|\kappa(x,y)$ for all $x,y\in [-1,1]$ --------(1). For $x=1,y=1-\frac{1}{n}$, $\rho_{_u}(x,y)=\frac{1}{n}$ and $\kappa(x,y)=2$, for $ n>2 $. Thus from (1), it follows that $\frac{1}{2n}\geq |\alpha|$, for all $n>2$ which implies $\alpha =0$ ------ a contradiction. Hence $\rho_{_u}\not\in L(\kappa)$. 
	 	
	 	Thus $\downarrow\kappa\smallsetminus\{O\}\not\subseteq L(\kappa)$. So by Definition \ref{d:feasible}, we can say that $\kappa\not \in Q\Big(\mathcal{D}\big([-1,1]\big)\Big)$.\end{proof}
 	
	 \begin{Rem}
	 	$\kappa,\rho_{_u}\in D\smallsetminus\{O\}$ with $\rho_{_u}\leq \kappa$. Hence $\kappa\in L(\rho_{_u})$. Therefore $D$ is an orderly dependent set. Consequently, $D$ is not a basis of $\mathcal{D}([-1,1])$.
	 \end{Rem}
 
 In the following theorem, we shall illustrate a certain scenario when basis of $ \mathcal D(\mathbf X)\smallsetminus\{O\} $ can be constructed. It will depict how the metrics can be compared with the
 basis elements (in view of Definitions \ref{d:gen} and \ref{d:basis}).
 
	\begin{Th}
		Suppose there exists a metric $d\in\mathcal{D}(\mathbf{X})$ such that $\displaystyle\inf_{\underset{x\not=y}{x,y\in\mathbf{X}}}d(x,y)=0$. If $ Q\big(\mathcal{D}(\mathbf X)\big)\bigcap[\rho]\neq\emptyset,\forall\,\rho\in D $, where $ [\rho]:=\big\{e\in D:L(e)=L(\rho)\big\} $, then $ \mathcal{D(\mathbf X)}\smallsetminus\{O\} $ has a basis. {\em[}Here $ D $ is as in Theorem \ref{t:generator}{\em]}
\label{t:basis}	\end{Th}
	
\begin{proof}
	By Theorem \ref{t:generator}, $ D $ is nonempty and a generator of $ \mathcal{D(\mathbf X)}\smallsetminus\{O\} $. We define a relation $ \sim $ on $ D $ as follows : For $ \rho,e\in D $, $ \rho\sim e $ if $ L(\rho)=L(e) $. Then $ \sim $ is an equivalence relation on $ D $. The equivalence class containing $ \rho $ is given by $ [\rho]:=\big\{e\in D:L(e)=L(\rho)\big\} $. By hypothesis, $ Q\big(\mathcal{D}(\mathbf X)\big)\bigcap[\rho]\neq\emptyset,\forall\,\rho\in D $. Let us construct a set $ B $ taking exactly one element from each equivalence class in such a way that they belong to $ Q\big(\mathcal{D}(\mathbf X)\big) $ also. By axiom of choice $ B $ is nonempty. We claim that $ B $ is a basis of $ \mathcal{D(\mathbf X)}\smallsetminus\{O\} $.
	
	Let $ \rho\in\mathcal{D(\mathbf X)}\smallsetminus\{O\} $. Since $ D $ is a generator, $ \exists\,\rho'\in D $ such that $ \rho\in L(\rho') $. Again for $ \rho'\in D,\exists\,e\in B $ such that $ L(\rho')=L(e) $. Thus $ \rho\in L(e) $. So $ B $ generates $ \mathcal{D(\mathbf X)}\smallsetminus\{O\} $.
	
	If possible let $ B $ is not orderly independent. Then $ \exists\,e_1,e_2\in B $ such that $ e_1,e_2 $ are orderly dependent. Without any loss of generality we assume that $ e_1\in L(e_2) $. Then $\exists\,\alpha>0 $ such that $ e_1\geq \alpha e_2 $. Since $ e_1\in Q\big(\mathcal{D}(\mathbf X)\big)$, we have by Lemma \ref{l:Q(X)} $ L(\alpha e_2)=L(e_1) $ and then by Proposition \ref{p:L(x)}(iii) $ L(e_2)=L(\alpha e_2)=L(e_1) $ ------ which is not possible since $ e_1,e_2 $ being two distinct elements of $ B $ belong to two different equivalence classes. So $ B $ is orderly independent and hence is a basis of $ \mathcal{D(\mathbf X)}\smallsetminus\{O\} $. Moreover, $ \dim\mathcal{D(\mathbf X)}=[\text{Card}(B):0] $.
\end{proof}	

\begin{Rem}
	If $ \mathcal{D}(\mathbf X)\smallsetminus\{O\} $ possesses a basis $ B $ then any metric on $ \mathbf X $ can be compared with some metric within $ B $, but no two metrics within $ B $ are comparable; here comparability is judged in terms of the orderly dependence of $ \mathcal{D}(\mathbf X) $. If two metrics $ d,\rho\in\mathcal{D}(\mathbf X) $ satisfy the condition $ d\in L(\rho) $ and $ \rho\in L(d) $ \big[or equivalently $ L(d)=L(\rho) $, by Proposition \ref{p:L(x)} \big] then the two metrics must be equivalent as well as behave alike in terms of completeness. In the next section we shall measure the comparability of two metrics.
\end{Rem}

\section{Comparability of metrics using comparing function}

In this section we shall compare two metrics using `\textit{comparing function}'. This is a notion used in topological exponential vector space to measure the degree of comparability of two elements. Using this concept we shall characterise orderly independence of two elements in a topological exponential vector space. We shall compute the comparing function in $ \mathcal{D(\X)} $ and use it to determine orderly dependency of some standard metrics.

\begin{Def}\cite{norm}
	Let $X$ be a zero primitive evs over the field $\K$ of real or complex numbers with additive identity $\theta$. For $x,y\in X$, we define the \textit{comparing spectrum of $y$ relative to $x$} by $\sigma_{x}(y):=\big\{\lambda\in\K:$ $\lambda x\leq y\big\}$.
	
	Since $X$ is zero primitive, $y\geq\theta$ for all $y\in X$ and hence $0\in\sigma_x(y)$ so that it is nonempty.
\end{Def}

\begin{Prop}{\em\cite{bal}}
	For a topological evs $X$, $\sigma_x(y)$ is bounded, $\forall\,x\neq\theta,\forall\,y\in X$.
	\label{p:specbdd}\end{Prop}

\begin{Def}{\cite{bal}}
	Let $X$ be a zero primitive topological evs over the field $\K$ of real or complex numbers. For any $x\,(\neq\theta)\in X$, the \textit{comparing function relative to $x$}, denoted by $C_{x}$, is defined as\\ \centerline{$C_{x}(y):=\underset{\lambda\in\sigma_{x}(y)}{\sup}|\lambda|$, for all $y\in X$.}
	
	By Proposition \ref{p:specbdd}, $C_x$ is well-defined and $C_x(y)\in[0,\infty), \forall\,y\in X,\forall\,x\neq\theta$.
\end{Def}

\begin{Prop}{\em\cite{norm}}
	Let $X$ be a zero primitive topological evs over the field $\K$. For any non-zero homogeneous element $x\in X$ and any $y\in X$, $C_x(y)x\leq y$.
	\label{p:leq}\end{Prop}

\begin{Res}
	In a zero-primitive homogeneous topological evs $ X $ over the field $ \K $, for any $ x\neq\theta $, $ y\in L(x) $ iff $ C_x(y)>0 $.
\end{Res}

\begin{proof}
	If $ y\in L(x) $ then $\exists
	\,\alpha\in\K^* $ such that $ y\geq\alpha x $ which implies $ C_x(y)\geq|\alpha|>0 $. Conversely, if $ C_x(y)>0 $ then by Proposition \ref{p:leq}, $ C_x(y)x\leq y $ and hence  $ y\in L(x) $.
\end{proof}

This immediately implies the following Theorem.

\begin{Th}
	In a zero-primitive homogeneous topological evs $ X $ over the field $ \K $,\\ (i) two points $ x,y\in X\smallsetminus\{\theta\} $ will be orderly dependent iff either $ C_x(y)>0 $ or $ C_y(x)>0 $.\\ (ii) $ x,y\in X\smallsetminus\{\theta\} $ will be orderly independent iff $ C_x(y)=0=C_y(x) $.\\ (iii) for $ x,y\in X\smallsetminus\{\theta\} $, $ C_x(y)>0\Leftrightarrow y\in L(x)\Leftrightarrow L(y)\subseteq L(x) $.\\ (iv) for $ x,y\in X\smallsetminus\{\theta\} $, $ L(x)=L(y) $ iff $ C_x(y)C_y(x)>0 $.\\ (v) for $ x,y\in X\smallsetminus\{\theta\} $, $ C_x(y)=0=C_y(x) $ $ \Leftrightarrow L(x) ,L(y)$ are incomparable.
\label{t:Cxy}\end{Th}

We shall now compute the comparing function on the zero-primitive homogeneous topological evs $ \mathcal{D(\mathbf X)} $.

\begin{Th}
	Let $d\in\mathcal{D}(\mathbf{X})\smallsetminus\{O\}$. Then the comparing function relative to $d$ is\\ \centerline{$C_{d}(\rho)=\inf\limits_{\substack{x,y\in\mathbf{X}\\x\neq y}}\frac{\rho(x,y)}{d(x,y)}$, $\forall\,\rho\in\mathcal{D}(\mathbf{X})$.}
\end{Th} 

\begin{proof}
	Here $\sigma_{d}(\rho)=\big\{\lambda\in\K:\lambda\cdot d\leq \rho\big\}=\big\{\lambda\in\K:|\lambda|d(x,y)\leq\rho(x,y),\forall\,x,y\in\mathbf{X}\big\}$.\\ 	Let $s:=\inf\limits_{\substack{x,y\in\mathbf{X}\\x\neq y}}\frac{\rho(x,y)}{d(x,y)}$. Clearly $s\in [0,\infty)$. We prove that $C_{d}(\rho)=s$. 
	
	By definition of $s$, $s\leq\frac{\rho(x,y)}{d(x,y)}$ for all $x,y\in\mathbf{X}$ with $x\neq y$. Therefore $s\cdot d(x,y)\leq \rho(x,y)$ for all $x,y\in\mathbf{X}$. This implies $s\in\sigma_{d}(\rho)$. Therefore $s\leq\underset{\lambda\in\sigma_{d}(\rho)}{\sup} |\lambda|$ $= C_{d}(\rho)$.
	
	To prove the converse part, let $\lambda\in\sigma_{d}(\rho)$. Thus $\lambda\cdot d\leq \rho$ which implies $|\lambda| d(x,y)\leq \rho(x,y)$ for all $x,y\in\mathbf{X}$. So $|\lambda|\leq\frac{\rho(x,y)}{d(x,y)}$ for all $x,y\in\mathbf{X}$ with $x\neq y$ $\implies |\lambda|\leq \inf\limits_{\substack{x,y\in\mathbf{X}\\x\neq y}}\frac{\rho(x,y)}{d(x,y)}=s$, for any $\lambda\in\sigma_{d}(\rho)$. Hence $\underset{\lambda\in\sigma_{d}(\rho)}{\sup}|\lambda|\leq s$. Therefore $C_{d}(\rho)=s=\inf\limits_{\substack{x,y\in\mathbf{X}\\x\neq y}}\frac{\rho(x,y)}{d(x,y)}$, for all $\rho\in\mathcal{D}(\mathbf{X})$.
\end{proof}

\begin{Rem}
	Let $ \rho,d\in\mathcal{D(\mathbf X)}\smallsetminus\{O\} $. Then $ L(\rho)=L(d) $ iff $ C_{\rho}(d)\cdot C_d(\rho)>0 $, by Theorem \ref{t:Cxy}(iv),  and hence $ \rho,d $ induce the same topology. Moreover, $ L(\rho)=L(d)\Leftrightarrow $ $ C_{\rho}(d)\rho\leq d\leq C_d(\rho)^{-1}\rho $ which implies that the metric spaces $ (\mathbf X,\rho), (\mathbf X,d) $ behave alike in terms of completeness.
\end{Rem}

We have shown in Result \ref{r:usudis} that the discrete metric $ \rho_{_d} $ and the usual metric $ \rho_{_u} $ on $ \R $ are orderly independent. We now check this using comparing function.

\begin{Res} (i)
	$ C_{\rho_{_{d}}}(\rho_{_u})=0=C_{\rho_{_{u}}}(\rho_{_d}) $
	
	(ii) For any bounded metric $ \rho $ on any non-empty set $ \mathbf X $, $ C_{\rho_{_{\min}}}(\rho)\cdot C_{\rho}(\rho_{_{\min}})>0 $
\end{Res}

\begin{proof}(i)
	$ C_{\rho_{_{d}}}(\rho_{_u})=\inf\limits_{\substack{x,y\in\R\\x\neq y}}\frac{\rho_{_{u}}(x,y)}{\rho_{_{d}}(x,y)}=\inf\limits_{\substack{x,y\in\R\\x\neq y}}\frac{|x-y|}{1}=0 $ and\\ $ C_{\rho_{_{u}}}(\rho_{_d})=\inf\limits_{\substack{x,y\in\R\\x\neq y}}\frac{\rho_{_{d}}(x,y)}{\rho_{_{u}}(x,y)}=\inf\limits_{\substack{x,y\in\R\\x\neq y}}\frac{1}{|x-y|}=0 $.\\
	Then by Theorem \ref{t:Cxy}(ii), $ \rho_{_{d}},\rho_{_u} $ are orderly independent.\\	
	(ii) $ \frac{\rho(x,y)}{\rho_{_{\min}}(x,y)}=\frac{\rho(x,y)}{\min\{1,\rho(x,y)\}}=
	\begin{cases}
		1,\text{ if }0<\rho(x,y)\leq1\\
		\rho(x,y),\text{ if }\rho(x,y)>1
	\end{cases} $\\
Thus $ C_{\rho_{_{\min}}}(\rho)=\inf\limits_{\substack{x,y\in\mathbf X\\x\neq y}}\frac{\rho(x,y)}{\rho_{_{\min}}(x,y)}=\max\left\{1,\inf\limits_{\substack{x,y\in\mathbf X\\x\neq y}}\rho(x,y)\right\}\geq1>0 $ and\\ $ C_\rho({\rho_{_{\min}}})=\inf\limits_{\substack{x,y\in\mathbf X\\x\neq y}}\frac{\rho_{_{\min}}(x,y)}{\rho(x,y)}=\min\left\{1,\inf\limits_{\substack{x,y\in\mathbf X\\x\neq y}}\frac{1}{\rho(x,y)}\right\}=\min\big\{1,\frac{1}{M}\big\}>0 $,\\ where $ M:=\sup\{\rho(x,y):x,y\in\mathbf X,x\neq y\}>0,\ \rho $ being bounded.\\		
	Then by Theorem \ref{t:Cxy}(i), $ \rho_{_{\min}},\rho $ are orderly dependent, as we have already proved in Theorem \ref{t:min}.
\end{proof}

Thus the comparing function measures the comparability of two metrics on a set $ \mathbf X $.

\section{Comparability of norms on a linear space}

We shall explore, in this section, the techniques developed in the previous sections to investigate the comparability of two norms on a linear space. Here comparability of two norms is judged by the concept of orderly dependence; so to explore comparability of norms we need to investigate some basis of $ \norm\smallsetminus\{O\} $ (Example \ref{e:spofnorms}). Equivalence of any two norms on a finite dimensional linear space $ \X $ ensures the existence of a basis in $ \norm\smallsetminus\{O\} $. We shall prove that in an infinite dimensional linear space, a large number of orderly independent norms can be constructed depending on the dimension of the linear space.

\begin{Th}{\em\cite{norm}}
	Consider the topological evs $\mathcal{N}(\X)$ and let $f\,(\neq O)\in\mathcal{N}(\X)$. Then the comparing function relative to $f$ is $C_{f}(g)=\underset{x\in\X\smallsetminus\{\theta\}}{\inf}\frac{g(x)}{f(x)}$, $\forall\,g\in\mathcal{N}(\X)$. Here $\theta$ is the zero element in the vector space $\X$.
\label{t:Cnorm}\end{Th}

\begin{Th}
	If $\X$ is a finite dimensional linear space over the field $\K$, then $\norm\smallsetminus\{O\}$ has a basis.
\end{Th}

\begin{proof}
	Since $[\norm]_{0}=\{O\}$, the primitive space $[\norm]_{0}$ has no basis. Again since any two norms on a finite dimensional linear space are equivalent, for any two norms $ f,g\in\mathcal{N(\X)} $, $ \exists\,\alpha,\beta>0 $ such that $ \alpha f(x)\leq g(x)\leq\beta f(x),\forall\,x\in\X $. So $\alpha f\leq g\leq\beta f $. This shows that $ g\in L(f) $ and $ f\in L(g) $ and hence $ f,g $ are orderly dependent. Moreover, if we fix any norm $ f\in\mathcal{N(\X)} $ then $ g\in L(f),\forall\,g\in\mathcal{N(\X)}\smallsetminus\{O\} $. Thus $ \{f\} $ is a basis of $ \mathcal{N(\X)}\smallsetminus\{O\} $ and hence $ \dim\mathcal{N(\X)}=[1:0] $.
\end{proof}

We now discuss the existence of basis of $ \norm $ whenever $ \X $ is an infinite dimensional linear space.

\begin{Def}{\cite{norm}}
	Two norms on a linear space are said to be \textit{totally non-equivalent} if the topologies induced by them are totally incomparable.\label{d:total}
\end{Def}

\begin{Th}{\em\cite{norm}}
	Two norms $ f, g $ on a linear space $ \X $ will be totally non-equivalent
	iff $ C_f(g)=0=C_g(f) $.
\label{t:total}\end{Th}

Thus in view of Theorem \ref{t:total} and Theorem \ref{t:Cxy}(ii) we have the following result.

\begin{Res}
	Two norms on a linear space are orderly independent iff they are totally non-equivalent.\label{r:ordindpnorm}
\end{Res}

\begin{Cor}
	Let $\X$ be an infinite dimensional linear space over the field $\K$. If $B$ is a basis of $\norm\smallsetminus\{O\}$, then $B$ contains only totally non-equivalent norms.
\end{Cor}

We now discuss how many orderly independent norms can be defined on an infinite dimensional linear space. For this we need the following Lemma.

\begin{Lem}
	Let $ A $ be an infinite set. Then $ \exists\,B\subsetneqq A $ and disjoint countable subsets $ D_b,\,E_b $ of $ A $, for each $ b\in B $ such that $ |B|=|A| $ and $ A=B\displaystyle\sqcup\bigsqcup_{b\in B}(D_b\sqcup E_b) $. \big[Here $ |A| $ denotes the cardinality of $ A $ and $ \sqcup $ denotes disjoint union of sets.\big]
\label{l:card}\end{Lem}

\begin{proof}
	$ A $ being infinite, $ |A|=|A\times\{0,1\}|=|(A\times\{0\})\sqcup(A\times\{1\})| $.  So there is a bijection $ \varphi:A\to(A\times\{0\})\sqcup(A\times\{1\}) $. Put $ B:=\varphi^{-1}(A\times\{0\})  $ and $ C:=\varphi^{-1}(A\times\{1\}) $. Then $ A=B\sqcup C $ and $ |B|=|\varphi^{-1}(A\times\{0\})|=|A\times\{0\}|=|A|=|A\times\{1\}|=|\varphi^{-1}(A\times\{1\})|=|C| $.
	
	Now $ |C|=|B|=|B\times\N|,\ B$ being infinite. So there is a bijection $ \psi:C\to B\times\N $. Again $B\times\N=\displaystyle\bigsqcup_{b\in B}\Big(\big(\{b\}\times2\N\big)\sqcup\big(\{b\}\times(2\N+1)\big)\Big) $, where $ 2\N:=\{2n:n\in\N\} $ and $ 2\N+1:=\{2n+1:n\in\N\} $. Put $ D_b:=\psi^{-1}\big(\{b\}\times2\N\big) $ and $ E_b:=\psi^{-1}\big(\{b\}\times(2\N+1)\big) $ for each $ b\in B $. So $ C=\displaystyle\bigsqcup_{b\in B}(D_b\sqcup E_b) $. Therefore $ A=B\displaystyle\sqcup\bigsqcup_{b\in B}(D_b\sqcup E_b) $.
\end{proof}

\begin{Th}
	Let $\X$ be a linear space over the field $\K$ with $ \dim\X=\aleph $ (an infinite cardinal). Then there exists at least $2^\aleph\cdot\mathfrak{c}$ many pairs of orderly independent norms on $ \X $, $\mathfrak{c}$ being the cardinality of $ \R$.\label{t:pair}\end{Th}

\begin{proof}
	Let $H$ be a Hamel basis of $\X$. Then each $x\in\X$ has a unique representation $x=\displaystyle\sum_{h\in H}\lambda_{h}h$, where $\lambda_{h}\in\K$ and $\lambda_{h}=0$ except for finitely many $h$'s. 
	
	For any function $w: H \rightarrow(0, \infty)$, we can define another function $\|\cdot\|_w:\X\longrightarrow[0,\infty)$ by $\|x\|_{w}:=\displaystyle\max _{h \in H} w(h)\left|\lambda_{h}\right|$, if $ x $ has the representation as above. Clearly, $\|\cdot\|_w$ is a norm on $\X$.
	
	Since $H$ is infinite, we may write by Lemma \ref{l:card},	
	$ H=B\displaystyle\sqcup\bigsqcup_{t\in B}(D_t\sqcup E_t) $, where $ B\subsetneqq H $ with $ |B|=|H| $ and $ D_t,\,E_t $ are disjoint countable subsets of $ H $ with $D_{t}=\left\{d_{t, i} : i \in \N\right\},$  $E_{t}=\left\{e_{t,i} : i\in \N\right\}$, for each $ t\in B $.
	
	Let $\emptyset\subsetneqq C\subsetneqq B$ and $\gamma>1$. We now define a function
	$w_{_{C,\gamma}}: H \rightarrow(0, \infty)$ by
	
	$w_{_{C,\gamma}}(x):=\begin{cases}1, & x \in  B\\ \gamma^{i}, & x= d_{t, i}, \text{ where }t\in C \\
		1, & x=d_{t, i},\text{ where }t\not \in C\\ \gamma^{-i}, & x=e_{t, i},\text{ where }t \in C\\
		1, & x=e_{t, i},\text{ where }t\not \in C\end{cases}$\\	
	Then by above discussion, it follows that $\|\cdot\|_{C,\gamma}:=\|\cdot\|_{w_{_{C,\gamma}}}$ is a norm on $\X$. We now show that for $ (C,\gamma)\neq(C',\gamma') $ \Big[$\emptyset\subsetneqq C,C'\subsetneqq B$ and $\gamma,\gamma'>1$\Big] the norms $ \|\cdot\|_{C,\gamma}\text{ and }\|\cdot\|_{C',\gamma'} $ are orderly independent.
	
\noindent	\textbf{Case 1:} Let $C\not=C'$. Without loss of generality, we assume that $C \smallsetminus C' \neq \emptyset$. Let $ t\in C\smallsetminus C' $. 	
Then by Theorem \ref{t:Cnorm}, $ C_{_{\|\cdot\|_{C',\gamma'}}}(\|\cdot\|_{C,\gamma})=\displaystyle\inf_{x\in\X\smallsetminus\{\theta\}}\tfrac{\|x\|_{C,\gamma}}{\|x\|_{C',\gamma'}}\leq\lim _{i \rightarrow \infty} \tfrac{\left\|e_{t, i}\right\|_{C, \gamma}}{\left\|e_{t, i}\right\|_{C', \gamma'}}=\lim_{i\to\infty} \tfrac{\gamma^{-i}}{1}=0$ $(\because\gamma>1) $ and $ C_{_{\|\cdot\|_{C,\gamma}}}(\|\cdot\|_{C',\gamma'})=\displaystyle\inf_{x\in\X\smallsetminus\{\theta\}}\tfrac{\|x\|_{C',\gamma'}}{\|x\|_{C,\gamma}}\leq\lim _{i \rightarrow \infty} \tfrac{\left\|d_{t, i}\right\|_{C', \gamma'}}{\left\|d_{t, i}\right\|_{C, \gamma}}=\lim_{i\to\infty} \tfrac{1}{\gamma^i}=0\,(\because\gamma>1) $. So by Theorem \ref{t:Cxy}(ii), $ \|\cdot\|_{C,\gamma}\text{ and }\|\cdot\|_{C',\gamma'} $ are orderly independent.	
	
\noindent	\textbf{Case 2:} Let $C=C'$. Then $\gamma \neq \gamma'$.	Without loss of generality, we assume that $\gamma>\gamma'$. Then for any $ t\in C=C' $, $ C_{_{\|\cdot\|_{C',\gamma'}}}(\|\cdot\|_{C,\gamma})=\displaystyle\inf_{x\in\X\smallsetminus\{\theta\}}\tfrac{\|x\|_{C,\gamma}}{\|x\|_{C',\gamma'}}\leq\lim _{i \rightarrow \infty} \tfrac{\left\|e_{t, i}\right\|_{C, \gamma}}{\left\|e_{t, i}\right\|_{C', \gamma'}}=\lim_{i\to\infty} \tfrac{\gamma^{-i}}{(\gamma')^{-i}}=0$ $(\because\gamma>\gamma') $ and $ C_{_{\|\cdot\|_{C,\gamma}}}(\|\cdot\|_{C',\gamma'})=\displaystyle\inf_{x\in\X\smallsetminus\{\theta\}}\tfrac{\|x\|_{C',\gamma'}}{\|x\|_{C,\gamma}}\leq\lim _{i \rightarrow \infty} \tfrac{\left\|d_{t, i}\right\|_{C', \gamma'}}{\left\|d_{t, i}\right\|_{C, \gamma}}=\lim_{i\to\infty} \tfrac{(\gamma')^i}{\gamma^i}=0\,(\because\gamma>\gamma') $. So by Theorem \ref{t:Cxy}(ii), $ \|\cdot\|_{C,\gamma}\text{ and }\|\cdot\|_{C',\gamma'} $ are orderly independent.	
	
Thus $ \mathscr{A}:=\big\{\|\cdot\|_{C,\gamma}:\emptyset\subsetneqq C\subsetneqq B,\gamma>1\big\} $ is a collection of orderly independent norms on $ \X $. Now the cardinality of this set is $ 2^\aleph\cdot
\mathfrak{c} $, if $ |B|=|H|=\dim\X=\aleph $ (an infinite cardinal) and $ |\R|=\mathfrak{c} $.
\end{proof}

\begin{Cor}
	If $ \X $ is an infinite dimensional linear space over $ \K $ of dimension $ \aleph $ then dimension of the evs $ \mathcal{N}(\X) $ is $ [\,\widehat{\aleph}:0\,] $, where $ \widehat{\aleph} $ is at least $ 2^\aleph\cdot\mathfrak{c} $, provided $ Q\big(\mathcal{N(\X)}\big) $ generates $ \mathcal{N}(\X)\smallsetminus\{O\} $.
\end{Cor}

\begin{proof}
	Follows immediately from Theorems \ref{t:pair} and \ref{t:maxbasis}, considering the fact that $ [\mathcal{N(\X)}]_0=\{O\} $.
\end{proof}

\begin{Cor}
	If $ \X $ is an infinite dimensional linear space over $ \K $ of dimension $ \aleph $ then there are at least $ 2^\aleph\cdot\mathfrak{c} $ many totally non-equivalent norms on $ \X $.
\end{Cor}

\begin{proof}
	Immediately follows from Theorem \ref{t:pair} and Result \ref{r:ordindpnorm}.
\end{proof}

\noindent\textbf{Acknowledgment :} The first author is thankful to University Grants Commission, India
	for financial assistance.

\end{document}